\documentclass[reqno]{amsart}
\usepackage[latin9]{inputenc}
\usepackage{bm}
\usepackage{amstext}
\usepackage{amsthm}
\usepackage{amssymb}

\makeatletter
\numberwithin{equation}{section}
\numberwithin{figure}{section}
\newenvironment{lyxlist}[1]
{\begin{list}{}
{\settowidth{\labelwidth}{#1}
 \setlength{\leftmargin}{\labelwidth}
 \addtolength{\leftmargin}{\labelsep}
 }}
{\end{list}}

\usepackage{graphicx}
\usepackage{color}
\usepackage{latexsym}
\usepackage{amsfonts}\usepackage{amsthm}\usepackage{bm}

\def\subsection{\@startsection{subsection}{2}%
  \z@{.5\linespacing\@plus.7\linespacing}{.3\linespacing}%
  {\normalfont\bfseries}}

\theoremstyle{plain}
\newtheorem{lemma}{Lemma}[section]
\newtheorem{prop}[lemma]{Proposition}

\newtheorem{theorem}[lemma]{Theorem}
\newtheorem{corollary}[lemma]{Corollary}
\newtheorem{fact}[lemma]{Fact}
\newtheorem{claim}[lemma]{Claim}
\newtheorem*{titulemma*}{Titu's Lemma}
\newtheorem*{corollary*}{Corollary}

\theoremstyle{remark}
\newtheorem*{claim*}{\textmd{\textit{Claim}}}

\usepackage[mathscr]{eucal}

\def\R{\mathbb R}
\def\Z{\mathbb Z}

\DeclareMathOperator{\Position}{Pos}
\DeclareMathOperator*{\argmax}{arg\,max}
\DeclareMathOperator{\IxSetOfVec}{IndexSet}
\DeclareMathOperator{\length}{Length}
\DeclareMathOperator{\maxLength}{MaxLength}

\newcommand{\nop}{}
\newcommand{\Ivecs}{\ensuremath{\mathbb{I}}}
\newcommand{\ixv}{\ensuremath{\iota}}
\newcommand{\IxSetNotContaining}[1]{\overline{\mspace{1mu}\mathrm{I}\mspace{1mu}}(#1)}
\newcommand{\IxSetContaining}[1]{\mathrm{I}(#1)}
\newcommand{\IvecsNotInvolving}[1]{\overline{\mspace{1mu}\mathbb{I}\mspace{1mu}}(#1)}
\newcommand{\IvecsAllInvolving}[1]{\mathbb{I}(#1)}
\newcommand{\IvecsEpsNotInvolving}[1]{\overline{\mathbb{I}_\epsilon}(#1)}
\newcommand{\IvecsEpsAllInvolving}[1]{\mathbb{I}_\epsilon(#1)}
\newcommand{\downmapsto}{\rotatebox[origin=c]{-90}{\scalebox{1.5}{$\longmapsto$}\!\!\!\!}\mkern2mu}

\newcommand{\wmapsto}[1]{\,\overset{#1}{\scalebox{1.5}{$\longmapsto$}}\,}

\calclayout

\makeatother

\begin{document}

\title{String-Averaging Algorithms for Convex Feasibility with Infinitely
Many Sets}

\author{T. Yung Kong}

\address{Computer Science Department, Queens College, City University of New
York, Flushing, NY 11367, USA}

\email{tkong@qc.cuny.edu}

\author{Homeira Pajoohesh}

\address{Mathematics Department, Medgar Evers College, City University of
New York, Brooklyn, NY 11225, USA}

\email{hpajoohesh@mec.cuny.edu}

\author{Gabor T. Herman}

\address{Computer Science Ph.D. Program, The Graduate Center, City University
of New York, New York, NY 10016, USA}

\email{gabortherman@yahoo.com}
\begin{abstract}
\noindent Algorithms for convex feasibility find or approximate a
point in the intersection of given closed convex sets. Typically there
are only finitely many convex sets, but the case of infinitely many
convex sets also has some applications. In this context, a \emph{string}
is a finite sequence of points each of which is obtained from the
previous point by considering one of the convex sets. In a \emph{string-averaging}
algorithm, an iterative step from a first point to a second point
computes a number of strings, all starting with the first point, and
then calculates a (weighted) average of those strings' end-points:
This average is that iterative step's second point, which is used
as the first point for the next iterative step. For string-averaging
algorithms based on strings in which each point either is the projection
of the previous point on one of the convex sets or is equal to the
previous point, we present theorems that provide answers to the following
question: ``How can the iterative steps be specified so that the
string-averaging algorithm generates a convergent sequence whose limit
lies in the intersection of the sets of a given convex feasibility
problem?'' This paper focuses on the case where the given collection
of convex sets is infinite, whereas prior work on the same question
that we are aware of has assumed the collection of convex sets is
finite (or has been applicable only to a small subset of the algorithms
we consider). The string-averaging algorithms that are shown to generate
a convergent sequence whose limit lies in the intersection are also
shown to be perturbation resilient, so they can be superiorized to
generate sequences that also converge to a point of the intersection
but in a manner that is superior with respect to some application-dependent
criterion.

\bigskip{}

\end{abstract}

\maketitle

\section{Introduction}

\label{sec:intro} 

The \textit{convex feasibility problem} is that of finding or approximating
a point that belongs to every member of a collection of closed convex
sets. This paper is concerned with the application of projection-based
\emph{string-averaging} (a methodology that will be further described
in sec.~\ref{subsec:Notation-and-Definitions}) to the convex feasibility
problem. Whereas most work on convex feasibility has assumed that
the collection of convex sets has only finitely many members,\footnote{\label{fn:A-notable-exception}A notable exception in the context
of our results is the paper \cite{Comb97b}, which studies a class
of iterative schemes called \emph{EMOPP}. It is shown in \cite{Comb97b}
that, under certain conditions, an EMOPP iterative scheme will converge
to a solution of the convex feasibility problem for a finite or countably
infinite collection of convex sets in a Hilbert space. However, most
of the string-averaging iterative schemes considered in the present
paper do not belong to the EMOPP class: Indeed, a string-averaging
iterative scheme cannot belong to the EMOPP class unless the ``strings''
it uses all have length 1 (i.e., its weight functions satisfy condition
H2 of sec.~\ref{subsec:Outline} with $m=1$) and its active averaging
weights are bounded away from 0. The results of \cite{Comb97b} include
results (for EMOPP iterative schemes) that are similar in nature to
our first and third convergence theorems, but no results akin to our
fourth or fifth convergence theorems. } our primary focus will instead be on the case where the collection
is (countably) infinite.

Typical practical instances of the convex feasibility problem are
inverse problems. The point we wish to find or approximate may, for
example, be a point in $\R^{n}$ whose coordinates are the values
of $n$ numerical attributes (sometimes referred to as ``model parameters'')
of our model of some physical object, and the convex sets to which
the point is assumed to belong may be determined by measurements taken
of this object (sometimes referred to as ``observed data''). In
this context, the ``forward'' problem is to use our model to predict
the values of the observed data from values of the $n$ model parameters.
But our concern is the inverse of that problem: Given the values of
the observed data, find model parameters that are consistent with
the data. This is a convex feasibility problem when the $n$ model
parameters are regarded as the coordinates of a point $p$ in $\R^{n}$
and each observed data value constrains that point $p$ to lie in
a certain closed convex subset of $\R^{n}$. For examples of such
convex feasibility problems, see the papers \cite{CDH10,KSP09}. Much
other work that relates to convex feasibility can be found in the
literature on inverse problems.\footnote{At the time of writing, searching for the topic ``convex feasibility''
on the IOPScience website returns 30 articles, of which 24 were published
in the journal \emph{Inverse Problems}; a Google Scholar search for
``convex feasibility'' \emph{AND} ``inverse problems'' returns
many hundreds of articles.}

As mentioned above, prior work on the convex feasibility problem has
for the most part assumed the collection of convex sets is finite.
Nevertheless, there has been some interest in convex feasibility problems
with infinitely many convex sets. Indeed, the case of infinitely many
sets is considered in some early seminal papers on the subject published
more than two decades ago. For example, the 1967 paper of Gurin, Polyak
and Raik \cite{GPR67} begins with: ``MANY mathematical and applied
problems can be reduced to finding some common point of a system (finite
or infinite) of convex sets.'' Other examples are two 1997 papers
by Combettes \cite{Comb97,Comb97b}: The paper \cite{Comb97b} was
mentioned in footnote~\ref{fn:A-notable-exception}, while the paper
\cite{Comb97} specifies that the collection of convex sets in its
application of convex feasibility to image recovery may be ``finite
or countably infinite''. The usefulness of the infinite case is demonstrated
in \cite{YaOg04} by showing its applicability to set theoretic adaptive
filtering schemes for nonstationary random processes. A rather more
recent publication is \cite{Nedi10}, in which the iterative algorithm
for finding a point in the intersection of (infinitely many) given
closed convex sets selects the $k+1$st iterate by moving from the
$k$th iterate in the direction of the latter's projection on the
$j$th convex set, where $j$ is a stochastically selected index.
In a broader context, there has been significant interest in algorithms
that converge to common fixed points of infinite families of operators
satisfying various conditions. Results relating to such algorithms
are presented in, e.g., \cite{AlCe05,BaRZ18,BuRZ08,NiAb15,NiAb17,PhSu13,YYZ07}.

The case of infinitely many sets is useful in applications where there
is a potential infinity of samples or measurements, each one of which
gives rise to a convex set that contains the point we wish to recover.
For example, in \cite{BlHe06} the problem of locating an acoustic
source in a sensor network is formulated as a convex feasibility problem
and necessary and sufficient conditions are given under which, when
the number of samples increases to infinity, the convex feasibility
problem has a unique solution at the true source location. In \cite{KSP09},
it is shown that a particular algorithm of the kind we consider in
the present paper is applicable to the problem of estimating the orientation
distribution function (used to characterize crystallographic texture)
from diffraction data. In one experiment discussed in \cite{KSP09},
the algorithm is used to solve a convex feasibility problem in $\R^{1\thinspace372\thinspace000\thinspace000}$
in which the number of convex sets is potentially infinite; each convex
set is the set of points that satisfy a randomly generated equation,
and an iterative step is carried out as soon as a new equation is
available. The result reported in the paper for that experiment is
the one obtained after $1\thinspace000\thinspace000\thinspace000$
iterations. While our theorems have no immediate application to the
specific algorithm of \cite{KSP09}, the reported work demonstrates
that the algorithms we consider in the present paper can be used to
provide solutions even to enormously large problems.

\subsection{Notation and Definitions\label{subsec:Notation-and-Definitions}}

Let $C_{1},C_{2},C_{3},\dots$ be an infinite sequence of closed convex
subsets of $\R^{n}$ (where $n$ is an arbitrary but fixed positive
integer) such that $\bigcap_{i=1}^{\infty}C_{i}$ is nonempty and
at least one of the sets $C_{i}$ is not $\R^{n}$. We write $\mathcal{C}$
to denote the collection of all the members of the sequence $(C_{i})_{i=1}^{\infty}$.
Thus $\mathcal{C=}\{C_{i}\mid1\leq i<\infty\}$ and $\R^{n}\supsetneq\bigcap\mathcal{C}=\bigcap_{i=1}^{\infty}C_{i}\neq\emptyset$.
We use the term \emph{index} to mean a positive integer that should
be thought of as the index of one of the $C_{i}$s.

We write $\|\bm{v}\|$ to denote the Euclidean norm of a vector $\bm{v}$,
and for any $q\in\R^{n}$ and any nonempty set $Z\subseteq\R^{n}$
we write $d(q,Z)$ to denote the Euclidean distance between $q$ and
$Z$; thus $d(q,Z)=\inf_{z\in Z}\|q-z\|$.

Since every member of $\mathcal{C}$ is closed and convex, for each
$q\in\R^{n}$ and each index $j$ there will be a unique point $c\in C_{j}$
that minimizes $\|c-q\|$; we define $P_{j}:\R^{n}\rightarrow C_{j}$
to be the function that maps $q$ to that point of $C_{j}$. As a
simple example, if $n=2$ and $C_{6}$ is the $y$-axis in $\R^{n}=\R^{2}$,
then $P_{6}$ maps each point $(x,y)\in\R^{2}$ to the point $(0,y)\in C_{6}$.
We call $P_{j}$ a \emph{projection} operator. Note that $P_{j}(q)=q$
if and only if $q\in C_{j}$; an important special case of this is
that if $C_{i}=\R^{n}$ then $P_{i}$ is the identity map.

Let $\Ivecs$ be the set of all nonempty finite sequences of indices.
The members of $\Ivecs$ will be called \emph{index vectors}. Note
that $(i)$ is an index vector for every index $i$. If $\ixv=(\ixv_{1},\ixv_{2},\dots,\ixv_{m})$
is any index vector, then the integer $m$ will be denoted by $\length(\ixv)$
and the set of indices $\{\ixv_{1},\ixv_{2},\dots,\ixv_{m}\}$ that
occur in $\ixv$ will be denoted by $\IxSetOfVec(\ixv)$. The cardinality
of the set $\IxSetOfVec(\ixv)$ cannot exceed $\length(\ixv)$, and
is exactly $\length(\ixv)$ if and only if no two of the indices $\ixv_{1}$,
$\ixv_{2}$, \dots , $\ixv_{m}$ are equal. 

For any index vector $\ixv=(\ixv_{1},\ixv_{2},\dots,\ixv_{m})$ we
write $P[\ixv]$ to denote the operator $P_{\ixv_{m}}\circ P_{\ixv_{m-1}}\circ\dots\circ P_{\ixv_{1}}$;
thus $P[\ixv](x)=P_{\ixv_{m}}(P_{\ixv_{m-1}}(\dots(P_{\ixv_{1}}(x))\dots))\in C_{\ixv_{m}}$
for all $x\in\R^{n}$. 

The results of this paper relate to the following question: 
\begin{itemize}
\item How can we use operators of the form $P[\ixv]$, or weighted averages
of such operators, to iteratively generate a convergent sequence $x^{0},x^{1},x^{2},\dots\in\R^{n}$
whose limit lies in every member of $\mathcal{C}$?
\end{itemize}
Here $x^{0}$ is a user-specified ``seed'' point (which may be any
point in $\R^{n}$), and the points $x^{1},x^{2},\dots$ are to be
computed one at a time, with $x^{k+1}$ being computed from the preceding
point $x^{k}$ for $k=0,1,2,\dots$. Each point $x^{k}$ will be called
an \emph{iterate}, and the process of computing $x^{k+1}$ from $x^{k}$
will be referred to as \emph{iteration $k$}. In this context the
integer $k\geq0$ will be called the \emph{iteration number}. 

Our goal is to prove five convergence theorems. The first three of
these theorems will concern sequences $x^{0},x^{1},x^{2},\dots$ in
which each iterate $x^{k+1}$ is a weighted average of some finite
subset of the set of points $\{P[\ixv](x^{k})\mid\ixv\in\Ivecs\}$.
By this we mean that the iterates satisfy 
\begin{equation}
{\textstyle {x^{k+1}=\sum_{\ixv\in\Ivecs}w^{k}(\ixv)P[\ixv](x^{k})}\qquad\text{for all \ensuremath{k\geq0}}}\label{alg}
\end{equation}
for some functions $w^{0},w^{1},w^{2},\dots:\Ivecs\rightarrow[0,1]$
such that the following are true when $w$ is any $w^{k}$: 
\begin{lyxlist}{00.000.}
\item [{\qquad{}A.}] $w$ has finite support\textemdash i.e., $\{\ixv\in\Ivecs\mid w(\ixv)\neq0\}$
is a finite set. 
\item [{\qquad{}B.}] $\sum_{\ixv\in\Ivecs}w(\ixv)=1$. 
\end{lyxlist}
We will think of the sequence $x^{0},x^{1},x^{2},\dots$ defined by
(\ref{alg}) as being generated by a \textit{string-averaging }\textit{\emph{algorithm}}.
Such algorithms were first proposed in \cite{CEH01}, but only for
the case of finitely many convex sets. In this context, a \emph{string}
is a finite sequence of points, each one of which is obtained from
the previous point by considering exactly one of the convex sets.
A string-averaging algorithm computes an iterate $x^{k+1}$ by first
computing a finite number of strings, each starting with the preceding
iterate $x^{k}$, and then calculates a weighted average of the last
points of those strings; this weighted average is taken as $x^{k+1}$.
Although \cite{CEH01} did not consider the possibility of using different
weighted averaging operators at different iterations, other works
have commonly allowed this \cite{BaRZ18,CeZa13,CeZa15,Comb97b,NiAb15,ReZa16},
just as we will. Indeed, using the same weighted averaging operator
at every iteration would be unsatisfactory when the collection $\mathcal{C}$
is infinite, as it would ignore all but finitely many of the members
of $\mathcal{C}$.

Any function $w:\Ivecs\rightarrow[0,1]$ that satisfies conditions
A and B will be called a \emph{weight function} on $\Ivecs$ or just
a \emph{weight function}. If $w$ is a weight function and $\ixv$
an index vector such that $w(\ixv)>0$, we say $w$ \emph{uses} $\ixv$
and, for each index $i\in\IxSetOfVec(\ixv)$, we similarly say that
$w$ \emph{uses }$i$; we also refer to each member of the set $\{w(\ixv)\mid\ixv\in\Ivecs\}\setminus\{0\}$
as a \emph{weight} or an \emph{averaging weight }that is\emph{ used}
by $w$. In the context of equation (\ref{alg}), we call $w^{k}$
the weight function used at iteration $k$, and say iteration $k$
\emph{uses} an index vector $\ixv$, an index $i$, or a weight $\mathsf{w}$
if $w^{k}$ uses $\ixv$, $i$, or $\mathsf{w}$. If $k$ and $x^{k}$
are such that $w^{k}(\ixv)=\mathsf{w}>0$ and $P[\ixv]x^{k}\neq x^{k}$
for some index vector $\ixv$, then we say the weight $\mathsf{w}$
is \emph{active }at iteration $k$.

We do not assume the sets $C_{1},C_{2},C_{3},\dots$ are all known
in advance; they may, for example, be presented to the algorithm one
at a time. Similarly, we do not assume the weight functions $w^{0},w^{1},w^{2},\dots$
are fixed in advance. To compute $x^{k+1}$ from $x^{k}$ an algorithm
might, for example, compute $P[\ixv](x^{k})$ for various index vectors
$\ixv$ and then use the results to decide what weight function $w^{k}$
to use at iteration $k$. Moreover, the choice of $w^{k}$ need not
be deterministic.

Before describing our theorems in more detail, we mention two noteworthy
special cases of our results. The first special case, which we will
call the \emph{finite case}, is that in which there are only finitely
many indices $i$ for which $C_{i}\neq\R^{n}$. In the finite case
our convergence theorems can be understood as results about projection-based
string-averaging algorithms for a finite sequence of sets $C_{1},\dots,C_{K}$.

The second special case is that in which $C_{1}=\R^{n}$ and $w^{k}((1))>0$
for at least one iteration number $k$. In this special case let us
write $\lambda_{k}=1-w^{k}((1))$ for every iteration number $k$,
so that $\lambda_{k}=\sum_{\ixv\in\Ivecs\setminus\{(1)\}}w^{k}(\ixv)$.
Then our first three theorems may be construed as convergence theorems
for \emph{under-relaxed} versions of algorithms that generate sequences
specified by (\ref{alg}); $\lambda_{k}$ is the relaxation parameter
at iteration $k$. This is because we can rewrite (\ref{alg}) as
\[
{\textstyle {x^{k+1}=x^{k}+\lambda_{k}\left((\sum_{\ixv\in\Ivecs}\omega^{k}(\ixv)P[\ixv](x^{k}))-x^{k}\right)}\qquad\text{for all \ensuremath{k\geq0}},}
\]
where $\omega^{k}:\Ivecs\rightarrow[0,1]$ is a weight function such
that $\lambda_{k}\omega^{k}(\ixv)=w^{k}(\ixv)$ for $\ixv\in\Ivecs\setminus\{(1)\}$
and $\omega^{k}((1))=0$.

\subsection{Outline of the Main Results\label{subsec:Outline}}

Our first three theorems imply that certain conditions (on $\mathcal{C}=\{C_{i}\mid1\leq i<\infty\}$
or on the weight functions $w^{0},w^{1},w^{2},\dots$) are sufficient
for the sequence $x^{1},x^{2},x^{3}\dots$ defined by equation (\ref{alg})
to converge to a point in $\bigcap\mathcal{C}$ (regardless of what
point is chosen as the seed point $x^{0}$). 

In the finite case, the second and third of these theorems are generalizations
of the previously known fact that $x^{1},x^{2},x^{3}\dots$ converges
to a point in $\bigcap\mathcal{C}$ if the weight functions satisfy
the following three additional hypotheses (which are \emph{not} assumed
in this paper):

\begin{lyxlist}{00......}
\item [{\quad{}H1.}] \noindent There is a positive lower bound on the
nonzero averaging weights used by $(w^{k})_{k=0}^{\infty}$\textemdash i.e.,
there is some $\epsilon>0$ such that $w^{k}(\ixv)>\epsilon$ for
every $k\geq0$ and every $\ixv\in\Ivecs$ such that $w^{k}(\ixv)\neq0$.\smallskip{}
\item [{\quad{}H2.}] There is an upper bound on the lengths of all the
index vectors used by $(w^{k})_{k=0}^{\infty}$\textemdash i.e., there
is an integer $m$ such that, for every $k\geq0$, if $\ixv\in\Ivecs$
and $\length(\ixv)>m$ then $w^{k}(\ixv)=0$. \smallskip{}
\item [{\quad{}H3.}] There is a positive integer $s$ such that each index
$i$ for which $C_{i}\neq\R^{n}$ is used by at least one member of
every $s$ consecutive members of the sequence $(w^{k})_{k=0}^{\infty}$\textemdash i.e.,
there is a positive integer $s$ such that ${\bigcup_{k=\kappa}^{\kappa+s-1}\{\IxSetOfVec(\ixv)\mid\ixv\in\Ivecs\text{\,\ and\ \,}w^{k}(\ixv)>0\}\supseteq\{i\mid C_{i}\neq\R^{n}\}}$
for every integer $\kappa\geq0$. 
\end{lyxlist}
That $x^{1},x^{2},x^{3}\dots$ converges to a point in $\bigcap\mathcal{C}$
in the finite case if H1 \textendash{} H3 are all satisfied follows
from \cite[Thm. 4.1]{ReZa16}; under the additional hypothesis that
H3 holds with $s=1$, it also follows from \cite[Thm. 12]{CeZa13}.
Moreover, \cite[Prop. 3.8]{NiAb15} implies that in the finite case
$x^{1},x^{2},x^{3}\dots$ converges to a point in $\bigcap\mathcal{C}$
whenever there are iteration numbers $n_{0}<n_{1}<n_{2}<\dots$ such
that, if we replaced $w^{k}$ with $w^{n_{k}}$ in H1~\textendash ~H3,
then H1 and H2 would hold and H3 would hold with $s=1$.

But while the hypotheses H1 \textendash{} H3 may seem quite mild in
the finite case, when there are infinitely many indices $i$ for which
$C_{i}\neq\R^{n}$ the hypothesis H3 cannot be satisfied\textemdash as
each iteration uses only finitely many indices\textemdash and it would
be quite restrictive to assume H1 or H2. (Hypothesis H1 would exclude
the algorithm described in Example~1 of sec.~\ref{subsec:Examples},
while H2 would exclude the algorithms described in Examples 2 and
3 of sec.~\ref{subsec:Examples}.) Accordingly, whereas other authors
have frequently assumed one or more of H1 \textendash{} H3, none of
these hypotheses will be assumed in this paper.

Our first convergence theorem assumes instead that $\bigcap\mathcal{C}$
is an $n$-dimensional subset of $\R^{n}$\textemdash i.e., $\bigcap\mathcal{C}$
contains an $n$-dimensional simplex. Under this hypothesis it says
that, for any weight functions $w^{0},w^{1},w^{2},\dots$ and any
seed point $x^{0}$, the sequence $x^{1},x^{2},x^{3}\dots$ defined
by equation (\ref{alg}) converges, and its limit lies in every $C_{i}$
for which the sum (over all iterations) of the weights of the index
vectors that contain index $i$ is infinite. In particular, $x^{1},x^{2},x^{3}\dots$
must converge to a point in $\bigcap\mathcal{C}$ if the weight functions
satisfy the following condition: For every index $i$ such that $C_{i}\neq\R^{n}$,
the sum (over all iterations) of the weights of the index vectors
that contain index $i$ is infinite. (When H1 holds and H2 holds with
$m=1$, this fact is a special case of \cite[Thm. 5.3(i)]{Comb97b}.)

The conclusions of the first convergence theorem do not always hold
when the hypothesis that $\bigcap\mathcal{C}$ is an $n$-dimensional
subset of $\R^{n}$ is not satisfied. However, our second convergence
theorem gives sufficient conditions on $w^{0},w^{1},w^{2},\dots$
for $x^{1},x^{2},x^{3},\dots$ to converge to a point in $\bigcap\mathcal{C}$
even if the dimension of $\bigcap\mathcal{C}$ may be less than $n$.
These sufficient conditions may be thought of as a greatly weakened
version of the conditions H1 \textendash{} H3 above. We mention that
the convergence of standard methods due to Cimmino \cite[Algorithm 2.1]{CeEl02}
(assuming the relaxation parameters $\lambda_{k}$ are all in $(0,1]$
and satisfy $\sum_{k=0}^{\infty}\lambda_{k}=\infty$) and to Kaczmarz
\cite{Kacz37} for solving finite systems of linear equations follows
quite easily from the second convergence theorem. 

The convergence of these standard methods will follow even more easily
from our third convergence theorem, which assumes that the collection
$\mathcal{C}$ satisfies the following additional condition:
\begin{lyxlist}{00.00.0000}
\item [{$\qquad\qquad\qquad\qquad\quad\textbf{UASC}(\mathcal{C})$:}] $d\left(p,\bigcup\{C\in\mathcal{C}\mid p\not\in C\}\right)>0$
for all $p\in\R^{n}\setminus\bigcap\mathcal{C}$. 
\end{lyxlist}
The condition $\textbf{UASC}(\mathcal{C})$ is so called because it
holds just if the collection $\mathcal{C}$ has the property that
the \textbf{\textit{u}}nion of \textbf{\textit{a}}ny \textbf{\textit{s}}ubcollection
is \textbf{\textit{c}}losed. It is plain that $\textbf{UASC}(\mathcal{C})$
is always satisfied if $\mathcal{C}$ has only finitely many members,
and this is probably the case of greatest practical interest.

However, $\textbf{UASC}(\mathcal{C})$ may hold even if $\mathcal{C}$
has infinitely many members. For example, $\textbf{UASC}(\mathcal{C})$
holds if the $C$s form a descending chain in the sense that $C_{i}\supseteq C_{i+1}$
for every $i$. $\textbf{UASC}(\mathcal{C})$ also holds if each $C_{i}$
is the closed triangle in $\R^{2}$ whose vertices are $(-1,0)$,
$(1,0)$, and $(\cos\pi/2^{i},\sin\pi/2^{i})$. It is worth noting
that if $\mathcal{C}$ and $\mathcal{C}'$ are collections such that
$\textbf{UASC}(\mathcal{C})$ and $\textbf{UASC}(\mathcal{C}')$ hold,
then $\textbf{UASC}(\mathcal{C}\cup\mathcal{C}')$ holds too.

Our third convergence theorem says that if $\textbf{UASC}(\mathcal{C})$
holds then the conclusion of the first convergence theorem holds regardless
of the dimensionality of $\bigcap\mathcal{C}$: For any weight functions
$w^{0},w^{1},w^{2},\dots$ and any seed point $x^{0}$, the sequence
$x^{1},x^{2},x^{3}\dots$ defined by equation (\ref{alg}) converges,
and its limit lies in every $C_{i}$ for which the sum (over all iterations)
of the weights of the index vectors that contain index $i$ is infinite.
In particular, $x^{1},x^{2},x^{3}\dots$ must converge to a point
in $\bigcap\mathcal{C}$ if $\textbf{UASC}(\mathcal{C})$ holds and
there is no index $i$ such that $C_{i}\neq\R^{n}$ for which the
sum (over all iterations) of the weights of the index vectors that
contain index $i$ is finite. (When H1 holds and H2 holds with $m=1$,
and $\mathcal{C}$ has only finitely many members, this fact is a
special case of \cite[Thm. 5.3(v)]{Comb97b}.)

Our fourth convergence theorem concerns ``perturbed'' versions of
sequences $x^{0},x^{1},x^{2},\dots$ that satisfy equation (\ref{alg}).
More precisely, the theorem relates to sequences $x_{*}^{0},x_{*}^{1},x_{*}^{2},\dots\in\R^{n}$
that satisfy 
\begin{equation}
{\textstyle {x_{*}^{k+1}=\left(\sum_{\ixv\in\Ivecs}w^{k}(\ixv)P[\ixv](x_{*}^{k})\right)+\bm{v}^{k}}\qquad\text{for all \ensuremath{k\geq0}}}\label{algP}
\end{equation}
for some weight functions $w^{0},w^{1},w^{2},\dots:\Ivecs\rightarrow[0,1]$,
where each $\bm{v}^{k}$ is a vector in $\R^{n}$ that will be called
the \emph{perturbation vector} at iteration $k$. 

The theorem makes two assertions which together say that, if the series
$\sum_{i=0}^{\infty}\|\bm{v}^{k}\|$ converges, then the first three
convergence theorems remain valid if we substitute an arbitrary sequence
$x_{*}^{0},x_{*}^{1},x_{*}^{2},\dots$ that satisfies (\ref{algP})
for the sequence $x^{0},x^{1},x^{2},\dots$. If for every index $i$
such that $C_{i}\neq\R^{n}$ the sum (over all iterations) of the
weights of the index vectors that contain index $i$ is infinite,
and the hypotheses of one of the first three convergence theorems
are satisfied, then the fourth convergence theorem tells us that an
algorithm based on (\ref{alg}) has a property that is commonly referred
to as (\emph{bounded})\emph{ perturbation resilience}\textemdash see,
e.g., \cite{BaRZ18,CDH10,CeZa13,NiAb15,NiAb17,ReZa16}\textemdash with
respect to the problem of finding a point in $\bigcap\mathcal{C}$.
It follows that such algorithms are candidates for \emph{superiorization}
\cite{Cens15,CHJ17}, a methodology for improving the efficacy of
algorithms whose convergence is resilient to certain kinds of perturbation.

In the finite case, it is a previously known fact that our fourth
convergence theorem holds when the hypotheses H1 \textendash{} H3
above are all satisfied: That result follows from \cite[Thm. 4.5]{ReZa16};
under the additional hypothesis that H3 holds with $s=1$, it follows
from \cite[Thm. 12]{CeZa13} as well. Under the hypotheses stated
above in connection with \cite[Prop. 3.8]{NiAb15}, the result also
follows from that proposition. But we are mainly interested in the
case where $\mathcal{C}$ has infinitely many members, and so we do
not assume H1 \textendash{} H3.

Our fifth and last convergence theorem makes use of operators $P_{j}^{\epsilon}:\R^{n}\rightarrow\R^{n}$
that are defined for every $\epsilon>0$ and every index $j$ as follows:
\[
P_{j}^{\epsilon}(x)=\begin{cases}
P_{j}(x) & \text{if \ensuremath{\|x-P_{j}(x)\|\geq\epsilon}}\\
x & \text{if \ensuremath{\|x-P_{j}(x)\|<\epsilon}}
\end{cases}
\]

For any index vector $\ixv=(\ixv_{1},\ixv_{2},\dots,\ixv_{m})$, just
as $P[\ixv]$ denotes the composite operator $P_{\ixv_{m}}\circ P_{\ixv_{m-1}}\circ\dots\circ P_{\ixv_{1}}$,
we write $P^{\epsilon}[\ixv]$ to denote the operator $P_{\ixv_{m}}^{\epsilon}\circ P_{\ixv_{m-1}}^{\epsilon}\circ\dots\circ P_{\ixv_{1}}^{\epsilon}$.
The fifth convergence theorem deals with sequences that are generated
from a seed point in a way that is analogous to (\ref{algP}) but
with $P^{\epsilon}[\ixv]$ in place of $P[\ixv]$: It deals with sequences
$x_{\epsilon}^{0},x_{\epsilon}^{1},x_{\epsilon}^{2},\dots$ that satisfy
the condition 
\begin{equation}
{\textstyle {x_{\epsilon}^{k+1}=\left(\sum_{\ixv\in\Ivecs}w^{k}(\ixv)P^{\epsilon}[\ixv](x_{\epsilon}^{k})\right)+\bm{v}^{k}}\qquad\text{for all \ensuremath{k\geq0}}}\label{algEpsPerturb}
\end{equation}
for some weight functions $w^{0},w^{1},w^{2},\dots:\Ivecs\rightarrow[0,1]$
and perturbation vectors $\bm{v}^{0},\bm{v}^{1},\bm{v}^{2},\dots\in\R^{n}$.

The theorem says that if $\sum_{k=0}^{\infty}\|\bm{v}^{k}\|<\infty$
then, for any sequence of weight functions $w^{0},w^{1},w^{2},\dots$
and any seed point $x_{\epsilon}^{0}$, the sequence $x_{\epsilon}^{1},x_{\epsilon}^{2},x_{\epsilon}^{3},\dots$
defined by (\ref{algEpsPerturb}) converges, and its limit is within
distance $\epsilon$ of every $C_{i}$ for which the sum (over all
iterations) of the weights of the index vectors that contain index
$i$ is infinite. This is true regardless of the dimensionality of
$\bigcap\mathcal{C}$ and whether or not $\textbf{UASC}(\mathcal{C})$
holds.

Let us now assume that for every index $i$ such that $C_{i}\neq\R^{n}$
the sum (over all iterations) of the weights of the index vectors
that contain index $i$ is infinite. Then it follows from this fifth
convergence theorem that, for any positive integers $m$ and $N$
and any seed point $y^{0}\in\R^{n}$, we can compute successive points
in the sequence defined by (\ref{algEpsPerturb}) with some $\epsilon<1/N$
and $x_{\epsilon}^{0}=y^{0}$ until we find a point $y^{1}$ such
that $\max_{1\leq i\leq m}d(y^{1},C_{i})\leq1/N$. We can repeat this
process with $y^{1}$ as seed point (instead of $y^{0}$), new perturbation
vectors, and larger values of $m$ and $N$ to find a point $y^{2}$,
which we use as the next seed point. Repeating this process with larger
and larger values of $m$ and $N$, we obtain a sequence $y^{3},y^{4},y^{5},\dots$
that must in fact converge to a point that lies in every member of
$\mathcal{C}$, assuming the sum of the magnitudes of the perturbation
vectors used in computing this sequence is finite.\footnote{If $l_{m}$ is the number of iterations of (\ref{algEpsPerturb})
that are carried out to compute $y^{m+1}$ from $y^{m}$, and $(\bm{v}_{m}^{k})_{k=0}^{l_{m}-1}$
is the sequence of perturbation vectors that are used by those $l_{m}$
iterations, then what we are assuming here is that $\sum_{m=0}^{\infty}\sum_{k=0}^{l_{m}-1}\|\bm{v}_{m}^{k}\|<\infty$. }

\section{First Convergence Theorem}

In the rest of this paper, $w^{0},w^{1},w^{2},\dots$ will denote
arbitrary weight functions on $\Ivecs$, and $x^{0},x^{1},x^{2},\dots$
points in $\R^{n}$ that satisfy equation (\ref{alg}) for all iteration
numbers $k\geq0$. For each index $i$ the set of all the index vectors
in which $i$ occurs will be denoted by $\Ivecs\langle i\rangle$.
Thus $\Ivecs\langle i\rangle=\{\ixv\in\Ivecs\mid i\in\IxSetOfVec(\ixv)\}$.
For example, $(3,1,4,3,1,3,7)\in\Ivecs\langle i\rangle$ just if $i=1,3,4$,
or $7$. Note that $\Ivecs\langle i\rangle$ is an infinite set for
every index $i$.

Recall from the Introduction that $\mathcal{C}=\{C_{i}\mid1\leq i<\infty\}$,
where each $C_{i}$ is a closed convex subset of $\R^{n}$, $\bigcap_{i=1}^{\infty}C_{i}=\bigcap\mathcal{C}\neq\emptyset$,
and there is at least one index $i$ for which $C_{i}\neq\R^{n}$.
The main theorem of this section is:

\begin{theorem}[First Convergence Theorem] \label{firstCT} Suppose
$\bigcap\mathcal{C}$ is an $n$-dimensional subset of $\R^{n}$ (i.e.,
$\bigcap\mathcal{C}$ contains an $n$-dimensional simplex). Then
the sequence $x^{0},x^{1},x^{2},\dots$ converges, and its limit lies
in $C_{j}$ for every index $j$ such that $\sum_{k=0}^{\infty}\sum_{\ixv\in\Ivecs\langle j\rangle}w^{k}(\ixv)=\infty$.
In particular, the sequence $x^{0},x^{1},x^{2},\dots$ converges to
a point in $\bigcap\mathcal{C}$ if $\sum_{k=0}^{\infty}\sum_{\ixv\in\Ivecs\langle j\rangle}w^{k}(\ixv)=\infty$
for every index $j$ such that $C_{j}\neq\R^{n}$. \end{theorem}

This theorem will be proved in sec.~\ref{subsec-proof-of-1st}.

\subsection{Counterexamples When $\bigcap\mathcal{C}$ is Not $n$-Dimensional}

\label{subsec-counterexamples}

The conclusions of the First Convergence Theorem may be false if $\bigcap\mathcal{C}$
is not an $n$-dimensional subset of $\R^{n}$: Even if $\sum_{k=0}^{\infty}\sum_{\ixv\in\Ivecs\langle i\rangle}w^{k}(\ixv)=\infty$
for every index $i$, a sequence $(x^{k})_{k=0}^{\infty}$ that satisfies
(\ref{alg}) may fail to converge or may converge to a point that
is not in $\bigcap\mathcal{C}=\bigcap_{i=1}^{\infty}C_{i}$. 

Example 4.1 of \cite{Comb97b} provides a counterexample in which
$(x^{k})_{k=0}^{\infty}$ fails to converge even though $\sum_{k=0}^{\infty}\sum_{\ixv\in\Ivecs\langle i\rangle}w^{k}(\ixv)=\infty$
for every index $i$. That counterexample is an instance of a class
of counterexamples we now describe. 

Let $(\theta_{i})_{i=0}^{\infty}$ be any infinite sequence of real
numbers that are not all equal for which the following conditions
all hold:

\begin{lyxlist}{00.00}
\item [{\emph{\qquad{}}\hspace{0.25em}1.}] \noindent $0\leq|\theta_{i+1}-\theta_{i}|<\pi/2$
for $0\leq i<\infty$.
\item [{\emph{\qquad{}}\hspace{0.25em}2.}] As $i\rightarrow\infty$,
$|\theta_{i+1}-\theta_{i}|$ is $O(1/i)$ (which implies $\log\cos(\theta_{i+1}-\theta_{i})$
is $O(1/i^{2})$, by Taylor's Theorem). 
\item [{\emph{\qquad{}}\hspace{0.25em}3.}] For each index $i$ there
are infinitely many indices $j$ such that $\theta_{j}\bmod2\pi=\theta_{i}\bmod2\pi$.
\end{lyxlist}
Bearing in mind that $\sum_{i=1}^{\infty}1/i=\infty$, it is easy
to construct such sequences $(\theta_{i})_{i=0}^{\infty}$. For example,
writing $\phi_{i}$ for $\sum_{j=1}^{i}1/j$ , we can obtain such
a sequence $(\theta_{i})_{i=0}^{\infty}$ from the sequence $(\phi_{i})_{i=0}^{\infty}$
by inserting additional numbers in such a way that the resulting sequence
satisfies condition 3 as well as conditions 1 and 2.\footnote{One way to do this is to let $(\theta_{i})_{i=0}^{\infty}$ be the
nondecreasing sequence obtained from $(\phi_{i})_{i=0}^{\infty}$
by inserting the sequence of numbers $(\phi_{j}\bmod2\pi+(2j+1)2^{k}\pi)_{k=1}^{\infty}$
for each $j\geq0$. Then $(\theta_{i})_{i=0}^{\infty}$ satisfies
conditions 1 and 3. For each positive even integer $2m$ there exist
unique integers $j\geq0$ and $k\geq1$ such that $(2j+1)2^{k}=2m$
(as is clear when we write $2m$ in binary notation), and so the interval
$[2m\pi,2(m+1)\pi)$ contains exactly one of the inserted numbers.
It can be seen from this that the sequence $(\theta_{i})_{i=0}^{\infty}$
satisfies condition 2.}

Now let $R_{1},R_{2},R_{3},\dots\subset\R^{2}$ be the sequence of
closed rays starting from $(0,0)$ such that each ray $R_{i}$ passes
through the point $(\cos\theta_{i},\sin\theta_{i})$. (Thus $R_{i}=R_{j}$
if and only if $\theta_{i}\bmod2\pi=\theta_{j}\bmod2\pi$.) Let $x^{0}$
be the point $(\cos\theta_{0},\sin\theta_{0})$, and for $0\leq k<\infty$
let $x^{k+1}$ be the orthogonal projection of $x^{k}$ on $R_{k+1}$. 

Let $(C_{i})_{i=1}^{\infty}$ be the sequence obtained from $(R_{i})_{i=1}^{\infty}$
by omitting all but the first occurrence of each ray in $(R_{i})_{i=1}^{\infty}$,
so that $C_{i}\neq C_{j}$ whenever $i\neq j$ and $\{C_{i}\mid1\leq i<\infty\}=\{R_{i}\mid1\leq i<\infty\}$.
For $1\leq i<\infty$, let $\alpha(i)$ be the index such that $R_{i}=C_{\alpha(i)}$.
This implies $x^{k+1}=P_{\alpha(k+1)}(x^{k})$ for $0\leq k<\infty$
(where $P_{j}:\R^{n}\rightarrow C_{j}$ is the projection map, as
before). So if for every iteration number $k$ we define 
\begin{equation}
w^{k}(\ixv)=\begin{cases}
1 & \text{if \ensuremath{\ixv=(\alpha(k+1))}}\\
0 & \text{otherwise}
\end{cases}\label{eq:wt}
\end{equation}
then $(x^{k})_{k=0}^{\infty}$ satisfies equation (\ref{alg}). Moreover,
for every index $i$ it follows from (\ref{eq:wt}) that $\sum_{\ixv\in\Ivecs\langle i\rangle}w^{k}(\ixv)=1$
whenever $\alpha(k+1)=i$ (i.e., whenever $R_{k+1}=C_{i}$), and hence
that $\sum_{k=0}^{\infty}\sum_{\ixv\in\Ivecs\langle i\rangle}w^{k}(\ixv)=\infty$
(as condition~3 above implies there are infinitely many $k$s for
which $R_{k+1}=C_{i}$). However, $(x^{k})_{k=0}^{\infty}$ fails
to converge: Indeed, as $k\rightarrow\infty$ it follows from conditions
1 and 2 that $\|x^{k}\|=\prod_{i=0}^{k-1}\cos(\theta_{i+1}-\theta_{i})$
and that this product converges to a nonzero value (since $\sum_{i=1}^{\infty}1/i^{2}$
converges), while conditions 1 and 3 imply $\theta_{k}\bmod2\pi$
does not converge.

We can give a similar counterexample to show that a sequence satisfying
(\ref{alg}) may converge to a point that is \emph{not} in $\bigcap_{i=1}^{\infty}C_{i}$
even though $\sum_{k=0}^{\infty}\sum_{\ixv\in\Ivecs\langle i\rangle}w^{k}(\ixv)=\infty$
for every index $i$. Let $y^{0},y^{1},y^{2},\dots$ be a convergent
subsequence $x^{r_{0}},x^{r_{1}},x^{r_{2}},\dots$ of the sequence
$x^{0},x^{1},x^{2},\dots$ of the above counterexample, and redefine
the weight function $w^{k}$ for each $k$ so that $w^{k}(\ixv)=1$
if $\ensuremath{\ixv=(\alpha(r_{k}+1),\alpha(r_{k}+2),\dots,\alpha(r_{k+1}))}$
and $w^{k}(\ixv)=0$ otherwise. Now the convergent sequence $y^{0},y^{1},y^{2},\dots$
satisfies equation (\ref{alg}) with $y^{k+1}$ and $y^{k}$ in place
of $x^{k+1}$ and $x^{k}$. Much as before, we have that $\sum_{k=0}^{\infty}\sum_{\ixv\in\Ivecs\langle i\rangle}w^{k}(\ixv)=\infty$
for every index $i$. But if $y^{\infty}=\lim_{k\rightarrow\infty}y^{k}$ then
we have that $\|y^{\infty}\|=\lim_{k\rightarrow\infty}\|y^{k}\|=\lim_{k\rightarrow\infty}\|x^{k}\|>0$,
and so $y^{\infty}\not\in\bigcap_{i=1}^{\infty}C_{i}=\{(0,0)\}$.

In the above counterexamples the collection $\mathcal{C}=\{C_{i}\mid1\leq i<\infty\}$
has infinitely many members. We will see from the Third Convergence
Theorem that there are no such counterexamples in the finite case.

\subsection{Preliminary Results}

We now establish some basic properties of the projection operators
$P_{i}$ and the sequence $x^{0},x^{1},x^{2},\dots$ that will be
used to prove the First Convergence Theorem. In this subsection we
do \emph{not} assume $\bigcap\mathcal{C}$ is an $n$-dimensional
subset of $\R^{n}$, so the results we establish here can also be
used to prove other convergence theorems that do not have this hypothesis.

The results of this subsection depend on the following geometrical
fact:

\begin{lemma} \label{lem-geom} Let $u$, $v$, and $a$ be points
in $\R^{n}$ such that $(a-u)\cdot(v-u)>0$. Then the line segment
that joins $u$ to $v$ contains points that are closer to $a$ than
$u$ is. \end{lemma}

\begin{proof} When $0\leq\epsilon\leq1$ the point $u+\epsilon(v-u)$
lies on the closed line segment that joins $u$ to $v$, and for sufficiently
small positive $\epsilon$ it is readily confirmed that this point
is closer to $a$ than $u$ is: Indeed, the squares of the distances
from $a$ to $u$ and from $a$ to $u+\epsilon(v-u)$ are respectively
$\|a-u\|^{2}$ and $\|a-u-\epsilon(v-u)\|^{2}$, and we have that
$\|a-u\|^{2}-\|a-u-\epsilon(v-u)\|^{2}=2\epsilon(a-u)\cdot(v-u)-\epsilon^{2}\|v-u\|^{2}>0$
when $\epsilon$ is sufficiently small and positive (since $(a-u)\cdot(v-u)>0$).
\end{proof}

\begin{prop}\label{fact-one} Let $i$ be an index such that $C_{i}\neq\R^{n}$,
let $p\in C_{i}$, and let $z\in\R^{n}$. Then: 
\begin{align*}
\|P_{i}(z)-p\|^{2} & \ \leq\ \|z-p\|^{2}-\|P_{i}(z)-z\|^{2}-2d(p,\R^{n}\setminus C_{i})\,\|P_{i}(z)-z\|\\
 & \ =\ \|z-p\|^{2}-d(z,C_{i})^{2}-2d(p,\R^{n}\setminus C_{i})\,d(z,C_{i})
\end{align*}
\end{prop}
\begin{proof}
We have that $\|P_{i}(z)-p\|^{2}=\|z-p\|^{2}-\|P_{i}(z)-z\|^{2}-2(P_{i}(z)-p)\cdot(z-P_{i}(z))$
because: 
\[
\|z-p\|^{2}\ =\ \|(P_{i}(z)-p)+(z-P_{i}(z))\|{}^{2}\ =\ \|P_{i}(z)-p\|^{2}+\|z-P_{i}(z)\|^{2}+2(P_{i}(z)-p)\cdot(z-P_{i}(z))
\]
So it remains only to show that: 
\begin{equation}
(P_{i}(z)-p)\cdot(z-P_{i}(z))\ \geq\ d(p,\R^{n}\setminus C_{i})\,\|z-P_{i}(z)\|\label{eq:to-be-shown}
\end{equation}
Note that $(P_{i}(z)-p)\cdot(z-P_{i}(z))\geq0$, for otherwise it
would follow from Lemma~\ref{lem-geom} (on putting $a=z$, $u=P_{i}(z)$,
and $v=p$) that the line segment joining $p$ to $P_{i}(z)$ contains
points that are strictly closer to $z$ than $P_{i}(z)$ is, which
would contradict the definition of $P_{i}(z)$ since the line segment
lies in the convex set $C_{i}$. 

If $z\in C_{i}$, then $P_{i}(z)=z$ and (\ref{eq:to-be-shown}) is
true. Now let us assume $z\not\in C_{i}$. Let $y$ be the orthogonal
projection of $p$ on the hyperplane through $P_{i}(z)$ that is orthogonal
to $z-P_{i}(z)$. Then $(z-P_{i}(z))\cdot(y-P_{i}(z))=0$, and so
$(z-P_{i}(z))\cdot(y+\lambda(z-P_{i}(z))-P_{i}(z))>0$ for all $\lambda>0$.
By Lemma~\ref{lem-geom}, this implies that for all $\lambda>0$
the line segment with endpoints $P_{i}(z)$ and $y+\lambda(z-P_{i}(z))$
contains points that are strictly closer to $z$ than $P_{i}(z)$
is. Such points are not in $C_{i}$ (as $P_{i}(z)$ is the point in
$C_{i}$ that is closest to $z$), so we see from the convexity of
$C_{i}$ (and the fact that $P_{i}(z)\in C_{i}$) that $y+\lambda(z-P_{i}(z))\notin C_{i}$
if $\ensuremath{\lambda>0}.$ Thus, there exist points arbitrarily
close to $y$ that are not in $C_{i}$, whence $d(p,\R^{n}\setminus C_{i})\leq\|y-p\|$.
Moreover, it follows from the definition of $y$ that the vector projection
of $P_{i}(z)-p$ on $z-P_{i}(z)$ is $y-p$, and we have seen that
$(P_{i}(z)-p)\cdot(z-P_{i}(z))\geq0$. Hence $(P_{i}(z)-p)\cdot\frac{z-P_{i}(z)}{\|z-P_{i}(z)\|}=\|y-p\|\geq d(p,\R^{n}\setminus C_{i})$,
which implies (\ref{eq:to-be-shown}).
\end{proof}
We now deduce a few lemmas from Proposition~\ref{fact-one}.

\begin{lemma} \label{gen-fact-zero} Let $\ixv=(\ixv_{1},\dots,\ixv_{m})\in\Ivecs$
be such that $\bigcap_{j\in\IxSetOfVec(\ixv)}C_{j}=\bigcap_{j=1}^{m}C_{\ixv_{j}}\neq\R^{n}$,
let $s\in\bigcap_{j=1}^{m}C_{\ixv_{j}}$, and let $z\in\R^{n}$. Then: 
\begin{lyxlist}{00.000}
\item [{\emph{\qquad{}}\hspace{0.25em}1.}] $\|P[\ixv](z)-s\|\leq\|z-s\|$,
with equality just if $z\in\bigcap_{j=1}^{m}C_{\ixv_{j}}$. 
\item [{\emph{\qquad{}}\hspace{0.25em}2.}] $\|P[\ixv](z)-s\|^{2}\leq\|z-s\|^{2}-2d(s,\R^{n}\setminus\bigcap_{j=1}^{m}C_{\ixv_{j}})\,\|P[\ixv](z)-z\|$. 
\item [{\emph{\qquad{}}\hspace{0.25em}3.}] $\|P[\ixv](z)-s\|^{2}\leq\|z-s\|^{2}-2d(s,\R^{n}\setminus\bigcap_{j=1}^{m}C_{\ixv_{j}})\,d(z,C_{\ixv_{q}})$
for $1\leq q\leq m$. 
\end{lyxlist}
\end{lemma}

\begin{proof} We see from Proposition~\ref{fact-one} that if $p\in C_{i}$
and $z\in\R^{n}$ then $\|P_{i}(z)-p\|\leq\|z-p\|$, with equality
just if $z\in C_{i}$. Assertion 1 follows from applying this observation
to each $P_{\ixv_{j}}$ in turn, since $s\in C_{\ixv_{j}}$ for $1\leq j\leq m$.

To prove assertions 2 and 3, we define $z_{0}=z$ and, for $1\leq r\leq m$,
we define $z_{r}=P[(\ixv_{1},\ixv_{2},\dots,\ixv_{r})](z)$, so that
$z_{m}=P[\ixv](z)$. Since 
\[
{\textstyle {d(s,\R^{n}\setminus\bigcap_{j=1}^{m}C_{\ixv_{j}})\leq d(s,\R^{n}\setminus C_{\ixv_{r}})}},
\]
the following is true for $1\leq r\leq m$: 
\begin{alignat}{2}
\|z_{r}-s\|^{2} & \ =\ \|P_{\ixv_{r}}(z_{r-1})-s\|^{2}\nonumber \\
 & \ \leq\ \|z_{r-1}-s\|^{2}-2d(s,\R^{n}\setminus C_{\ixv_{r}})\,\|z_{r}-z_{r-1}\| & \qquad & \text{by Proposition~\ref{fact-one}}\nonumber \\
 & \ \leq\ \|z_{r-1}-s\|^{2}-{\textstyle {2d(s,\R^{n}\setminus\bigcap_{j=1}^{m}C_{\ixv_{j}})\,\|z_{r}-z_{r-1}\|}}\label{eq:implies next}
\end{alignat}
Hence the following holds for $1\leq q\leq m$: 
\begin{align}
\|z_{q}-s\|^{2} & \ \leq\ \|z_{0}-s\|^{2}-{\textstyle {2d(s,\R^{n}\setminus\bigcap_{j=1}^{m}C_{\ixv_{j}})\,\sum_{r=1}^{q}\|z_{r}-z_{r-1}\|}\notag}\nonumber \\
 & \ \leq\ \|z_{0}-s\|^{2}-{\textstyle {2d(s,\R^{n}\setminus\bigcap_{j=1}^{m}C_{\ixv_{j}})\,\|z_{q}-z_{0}\|},}\label{eq:implies(ii)(iii)}
\end{align}
where the second inequality follows from the triangle inequality.
Since $z_{0}=z$ and $z_{m}=P[\ixv](z)$, we get assertion 2 from
the case $q=m$ of (\ref{eq:implies(ii)(iii)}).

Since $z_{q}\in C_{\ixv_{q}}$, we have that $\|z_{q}-z_{0}\|=\|z_{q}-z\|\geq d(z,C_{\ixv_{q}})$.
This and (\ref{eq:implies(ii)(iii)}) imply 
\[
\|z_{q}-s\|^{2}\leq\ \|z_{0}-s\|^{2}-{\textstyle {2d(s,\R^{n}\setminus\bigcap_{j=1}^{m}C_{\ixv_{j}})\,d(z,C_{\ixv_{q}})}\quad\text{for \ensuremath{1\leq q\leq m}.}}
\]
Assertion 3 follows, as $\|P[\ixv](z)-s\|=\|z_{m}-s\|\leq\|z_{q}-s\|$
for $1\leq q\leq m$ (by (\ref{eq:implies next})). \end{proof}

We see from equation (\ref{alg}), and conditions A and B of the definition
of a weight function on $\Ivecs$, that if $g:\R^{n}\rightarrow\R$
is any convex function then 
\begin{equation}
{\textstyle g(x^{k+1})\leq\sum_{\ixv\in\Ivecs}w^{k}(\ixv)g(P[\ixv](x^{k}))\qquad\text{for all \ensuremath{k\geq0}}}\label{Jensen}
\end{equation}
by Jensen's Inequality for functions on $\R^{n}$ \cite[Sec.~5.4]{Webs94-1}.

It is an easy consequence of the triangle inequality that $z\mapsto\|z-a\|$
is a convex function on $\R^{n}$ for any $a\in\R^{n}$. It follows
that $z\mapsto\|z-a\|^{2}$ is also a convex function on $\R^{n}$
(because $z\mapsto z^{2}$ is a convex function on $\R$). On applying
(\ref{Jensen}) to the function $g(z)=\|z-a\|^{2}$, we deduce that:
\begin{equation}
{\textstyle {\|x^{k+1}-a\|^{2}\leq\sum_{\ixv\in\Ivecs}w^{k}(\ixv)\|P[\ixv](x^{k})-a\|^{2}}\qquad\text{for all \ensuremath{k\geq0} and all \ensuremath{a\in\R^{n}}}}.\label{fin-alg}
\end{equation}
We use this to prove:

\begin{lemma} \label{lem-fact-one} Let $c\in\bigcap\mathcal{C}$.
Then for every iteration number $k$ we have that: 
\[
{\textstyle {\|x^{k+1}-c\|^{2}\leq\|x^{k}-c\|^{2}-2d(c,\R^{n}\setminus\bigcap\mathcal{C})\,\|x^{k+1}-x^{k}\|}}
\]
\end{lemma}

\begin{proof} If $\ixv=(\ixv_{1},\dots,\ixv_{m})$, then $d(c,\R^{n}\setminus\bigcap\mathcal{C})\leq d(c,\R^{n}\setminus\bigcap_{j=1}^{m}C_{\ixv_{j}})$.
Bearing this in mind, and writing $\delta$ for $d(c,\R^{n}\setminus\bigcap\mathcal{C})$,
we have that: 
\begin{alignat*}{2}
\|x^{k+1}-c\|^{2} & \,\leq\,{\textstyle {\sum_{\ixv\in\Ivecs}w^{k}(\ixv)\|P[\ixv](x^{k})-c\|^{2}}} &  & \text{by (\ref{fin-alg})}\\
 & \,\leq\,{\textstyle {\sum_{\ixv\in\Ivecs}w^{k}(\ixv)(\|x^{k}-c\|^{2}-2\delta\,\|P[\ixv](x^{k})-x^{k}\|)}} &  & \text{by statement~2 of Lemma~\ref{gen-fact-zero}}\\
 & \,=\,{\textstyle {\|x^{k}-c\|^{2}-2\delta\,\sum_{\ixv\in\Ivecs}w^{k}(\ixv)\|P[\ixv](x^{k})-x^{k}\|}} &  & \text{as \ensuremath{{\textstyle {\sum_{\ixv\in\Ivecs}w^{k}(\ixv)=1}}}}\\
 & \,\leq\,{\textstyle {\|x^{k}-c\|^{2}-2\delta\,\left\Vert \left(\sum_{\ixv\in\Ivecs}w^{k}(\ixv)P[\ixv](x^{k})\right)-x^{k}\right\Vert }} & \quad & \text{by (\ref{Jensen})}\\
 & \,=\,{\textstyle {\|x^{k}-c\|^{2}-2\delta\,\|x^{k+1}-x^{k}\|}}
\end{alignat*}
for all iteration numbers $k$. \end{proof}

\begin{corollary} \label{cor-bounded-motion} Let $c\in\bigcap\mathcal{C}$.
Then we have that 
\[
{\textstyle {\|x^{r}-c\|^{2}\leq\|x^{0}-c\|^{2}-2d(c,\R^{n}\setminus\bigcap\mathcal{C})\sum_{k=0}^{r-1}\|x^{k+1}-x^{k}\|}}
\]
for all iteration numbers $r$. \hfill{}\qedsymbol \end{corollary}

\begin{lemma} \label{lem-fact-another} Let $c\in\bigcap\mathcal{C}$.
Then for every iteration number $k$ we have that: 
\[
{\textstyle {\|x^{k+1}-c\|^{2}\leq\|x^{k}-c\|^{2}-2d(c,\R^{n}\setminus\bigcap\mathcal{C})\,d(x^{k},C_{j})\sum_{\ixv\in\Ivecs\langle j\rangle}w^{k}(\ixv)}\quad\text{for all indices \ensuremath{j}.}}
\]
\end{lemma}

\begin{proof} Let us write $\delta$ for $d(c,\R^{n}\setminus\bigcap\mathcal{C})$.
Then for every index $j$ we have that 
\begin{alignat*}{2}
\|x^{k+1}-c\|^{2} & \ \leq\ {\textstyle {\sum_{\ixv\in\Ivecs}w^{k}(\ixv)\|P[\ixv](x^{k})-c\|^{2}}} & \quad & \text{by (\ref{fin-alg})}\\
 & \ =\ {\textstyle {\sum_{\ixv\in\Ivecs\langle j\rangle}w^{k}(\ixv)\|P[\ixv](x^{k})-c\|^{2}}}\\
 & \qquad\qquad+{\textstyle {\sum_{\ixv\not\in\Ivecs\langle j\rangle}w^{k}(\ixv)\|P[\ixv](x^{k})-c\|^{2}}}\\
 & \ \leq\ {\textstyle {\sum_{\ixv\in\Ivecs\langle j\rangle}w^{k}(\ixv)(\|x^{k}-c\|^{2}-2d(x^{k},C_{j})\,\delta)}}\\
 & \qquad\qquad+{\textstyle {\sum_{\ixv\not\in\Ivecs\langle j\rangle}w^{k}(\ixv)\|x^{k}-c\|^{2}}} &  & \text{by statements 1 \& 3 of Lemma~\ref{gen-fact-zero}}\\
 & \ =\ {\textstyle {\|x^{k}-c\|^{2}-2d(x^{k},C_{j})\,\delta\sum_{\ixv\in\Ivecs\langle j\rangle}w^{k}(\ixv)}} &  & \text{as \ensuremath{{\textstyle {\sum_{\ixv\in\Ivecs}w^{k}(\ixv)=1}}}}
\end{alignat*}
for all iteration numbers $k$. \end{proof}

\begin{corollary} \label{cor-fact-another} For all indices $j$,
iteration numbers $r$, and $c\in\bigcap\mathcal{C}$ we have that
\[
{\textstyle {\|x^{r}-c\|^{2}\leq\|x^{0}-c\|^{2}-2d(c,\R^{n}\setminus\bigcap\mathcal{C})\sum_{k=0}^{r-1}\sum_{\ixv\in\Ivecs\langle j\rangle}d(x^{k},C_{j})\,w^{k}(\ixv).}}
\]
Hence $\sum_{k=0}^{\infty}\sum_{\ixv\in\Ivecs\langle j\rangle}d(x^{k},C_{j})\,w^{k}(\ixv)<\infty$
for every index $j$ if $d(c,\R^{n}\setminus\bigcap\mathcal{C})>0$
for some $c\in\bigcap\mathcal{C}$. \hfill{}\qedsymbol \end{corollary}

\subsection{Proof of the First Convergence Theorem}

\label{subsec-proof-of-1st}

Under the hypotheses of the theorem $\bigcap\mathcal{C}$ contains
an $n$-dimensional simplex. Let $c$ be the barycenter of some $n$-simplex
in $\bigcap\mathcal{C}$, so that $d(c,\R^{n}\setminus\bigcap\mathcal{C})>0$.
Then it follows from Corollary~\ref{cor-bounded-motion} (and the
fact that $\|x^{r}-c\|^{2}\geq0$ for all $r$) that the series $\sum_{k=0}^{\infty}\|x^{k+1}-x^{k}\|$
must converge.

Since $\|x^{r_{2}}-x^{r_{1}}\|\leq\sum_{k=r_{1}}^{r_{2}-1}\|x^{k+1}-x^{k}\|$
for all iteration numbers $r_{1}<r_{2}$ (by the triangle inequality),
and since $\sum_{k=0}^{\infty}\|x^{k+1}-x^{k}\|$ converges, $x^{0},x^{1},x^{2},\dots$
is a Cauchy sequence and must therefore converge.

Let $x^{\infty}$ be the limit of the sequence and $j$ any index
for which $x^{\infty}\notin C_{j}$, so that 
\begin{equation}
d(x^{k},C_{j})>d(x^{\infty},C_{j})/2>0\quad\text{for all sufficiently large \ensuremath{k}}.\label{eq:diverges}
\end{equation}
Then the series $\sum_{k=0}^{\infty}\sum_{\ixv\in\Ivecs\langle j\rangle}w^{k}(\ixv)$
converges, for if this series were divergent then it would follow
from (\ref{eq:diverges}) that the series $\sum_{k=0}^{\infty}\sum_{\ixv\in\Ivecs\langle j\rangle}d(x^{k},C_{j})\,w^{k}(\ixv)$
also diverges, contrary to Corollary~\ref{cor-fact-another}.

\section{Second Convergence Theorem}

As we saw in subsection~\ref{subsec-counterexamples}, the First
Convergence Theorem would not be true if we dropped the hypothesis
that $\bigcap\mathcal{C}$ is an $n$-dimensional subset of $\R^{n}$.
However, the convergence theorem we establish in the present section
will not depend on this hypothesis.

As before, $w^{0},w^{1},w^{2},\dots$ will denote arbitrary weight
functions on $\Ivecs$, and $x^{0},x^{1},x^{2},\dots$ points in $\R^{n}$
that satisfy equation (\ref{alg}) for all iteration numbers $k\geq0$.
For any weight function $w$, we define 
\[
\maxLength(w)=\max\{\length(\ixv)\mid\ixv\in\Ivecs\text{ and }w(\ixv)>0\}
\]
If $\ixv=(\ixv_{1},\ixv_{2},\dots,\ixv_{m})\in\Ivecs$ and $i\in\IxSetOfVec(\ixv)=\{\ixv_{1},\ixv_{2},\dots,\ixv_{m}\}$,
then we define $\Position(i,\ixv)$ to be the least $j$ for which
$\ixv_{j}=i$. If $i\not\in\IxSetOfVec(\ixv)$, then we say $\Position(i,\ixv)=\infty$.
As examples, if $\ixv=(9,8,2,4,1,6,2,1,2)$ then $\Position(1,\ixv)=5$,
$\Position(2,\ixv)=3$, $\Position(3,\ixv)=\infty$, and $\Position(4,\ixv)=4$.
Note that $\Position(i,\ixv)>0$, and that $\Position(i,\ixv)\leq\length(\ixv)$
whenever $\ixv\in\Ivecs\langle i\rangle$.

For any real numbers $f_{1},f_{2},f_{3},\dots\in[0,1]$ such that
$\sum_{r=1}^{\infty}f_{r}=\infty$, the following convergence theorem
gives a condition on the weight functions $w^{0},w^{1},w^{2},\dots$
that is sufficient to ensure that $x^{0},x^{1},x^{2},\dots$ converges
to a point in $\bigcap\mathcal{C}$. Provided $f_{1},f_{2},f_{3},\dots\in[0,1]$
and $\sum_{r=1}^{\infty}f_{r}=\infty$, the smaller the numbers $f_{1},f_{2},f_{3},\dots$
the better (i.e., the more general) the corresponding sufficient condition
will be. For example, the sufficient condition given by $f_{r}=\frac{1}{100r\log(r+1)}$
is rather better than the sufficient condition given by $f_{r}=\frac{1}{10r}$.

\begin{theorem}[Second Convergence Theorem] \label{2ndconv} Let
$f_{1},f_{2},f_{3},\dots\in[0,1]$ be numbers such that $\sum_{r=1}^{\infty}f_{r}=\infty$.
Suppose there is an infinite strictly increasing sequence of iteration
numbers $\kappa_{1}<\kappa_{2}<\kappa_{3}<\dots$ such that both of
the following are true: 
\begin{lyxlist}{00.000}
\item [{\textbf{\hspace{1em}C1}:}] ${\displaystyle {\sum_{k=\kappa_{r}}^{\kappa_{r+1}-2}\maxLength(w^{k})\leq\frac{1}{f_{r}}}}$
for all $r\in\Z^{+}$ such that $\kappa_{r+1}\geq\kappa_{r}+2$ and
$f_{r}>0$. 
\item [{\textbf{\hspace{1em}C2}:}] For every index $i$ such that $C_{i}\neq\R^{n}$,
${\displaystyle {\max_{\kappa_{r}\leq k<\kappa_{r+1}}\sum_{\ixv\in\Ivecs\langle i\rangle}\frac{w^{k}(\ixv)}{\Position(i,\ixv)}}\geq f_{r}}$
for all sufficiently large $r\in\Z^{+}$. 
\end{lyxlist}
Then $x^{0},x^{1},x^{2},\dots$ converges to a point in $\bigcap\mathcal{C}$.
\end{theorem}

As \textbf{C1} is vacuously satisfied if $\kappa_{r+1}=\kappa_{r}+1$
for all $r\in\Z^{+}$, the special case of this theorem in which $\kappa_{r}=r$
for all $r\in\Z^{+}$ can be stated more simply, as follows:

\begin{corollary} \label{cor-2ndconv} Let $f_{1},f_{2},f_{3},\dots\in[0,1]$
be such that $\sum_{r=1}^{\infty}f_{r}=\infty$, and suppose the following
is true for every index $i$ such that $C_{i}\neq\R^{n}$: 
\[
\sum_{\ixv\in\Ivecs\langle i\rangle}\frac{w^{r}(\ixv)}{\Position(i,\ixv)}\geq f_{r}\quad\text{for all sufficiently large \ensuremath{r\in\Z^{+}}}.
\]
Then $x^{0},x^{1},x^{2},\dots$ converges to a point in $\bigcap\mathcal{C}$.
\hfill{}\qedsymbol \end{corollary}

\subsection{Examples\label{subsec:Examples}}

Before proving the Second Convergence Theorem, we give simple examples
of weight functions for which the theorem implies convergence of the
sequence of iterates $x^{0},x^{1},x^{2},\dots$ to a point in $\bigcap\mathcal{C}$.

In the finite case, it is easy to see that if the weight functions
satisfy the hypotheses H1, H2, and H3 of sec.~\ref{subsec:Outline}
then \textbf{C1 }and \textbf{C2} hold for $(\kappa_{r})_{r=1}^{\infty}=s,2s,3s,4s,\dots$
and $(f_{r})_{r=0}^{\infty}=d,d,d,d,\dots$, where $s$ is any sufficiently
large positive integer and $d$ is any sufficiently small positive
constant. But in the finite case we will see from the Third Convergence
Theorem that the condition ``$\sum_{k=0}^{\infty}\sum_{\ixv\in\Ivecs\langle j\rangle}w^{k}(\ixv)=\infty$
for every index $j$ such that $C_{j}\neq\R^{n}$'' is sufficient
for $x^{0},x^{1},x^{2},\dots$ to converge to a point in $\bigcap\mathcal{C}$,
and this is a much more general sufficient condition than the ones
given by the Second Convergence Theorem. 

In the rest of this subsection we give three examples which do not
assume $\mathcal{C}$ has only finitely many members. The first two
examples are analogs for infinite collections of commonly used procedures
for converging to the intersection of a finite collection of closed
convex sets. In Example 1, $x^{k+1}$ is defined in a way that is
reminiscent of Cimmino's method \cite{Cimm38} of computing approximate
solutions of systems of linear equations (and especially of Algorithm~2.1
of \cite{CeEl02} with unity relaxation). Similarly, the iterates
of Example 2 are defined in a way that is reminiscent of Kaczmarz's
method \cite{Kacz37} of computing approximate solutions of linear
equation systems.

\subsubsection*{Example 1}

\noindent Suppose the sequence of iterates $x^{0},x^{1},x^{2},x^{3},\dots$
satisfies 
\[
x^{k+1}=\frac{1}{k}P_{1}(x^{k})+\frac{1}{k}P_{2}(x^{k})+\dots+\frac{1}{k}P_{k}(x^{k})\qquad\text{for all \ensuremath{k\geq1}.}
\]
When $r\geq1$, we have that $w^{r}(\ixv)>0$ just if $\ixv$ is one
of the $r$ index vectors $(1),(2),\dots,(r)$. For any index $i$,
when $r\geq i$ there is just one $\ixv\in\Ivecs\langle i\rangle$
for which $w^{r}(\ixv)>0$, namely $\ixv=(i)$, and we have that $w^{r}(\ixv)=1/r$
and $\Position(i,\ixv)=1$ for this $\ixv$. Hence $\sum_{\ixv\in\Ivecs\langle i\rangle}\frac{w^{r}(\ixv)}{\Position(i,\ixv)}=1/r$
if $r\geq i$, and so the hypotheses of Corollary~\ref{cor-2ndconv}
are satisfied for $f_{r}=1/r$. Therefore $(x^{k})_{k=0}^{\infty}$
converges to a point in $\bigcap\mathcal{C}$.

\subsubsection*{Example 2}

\noindent For every iteration number $k\geq1$, let $\ixv^{k}$ be
some permutation of $(1,2,\dots,k)$. Suppose the sequence of iterates
$x^{0},x^{1},x^{2},x^{3},\dots$ satisfies 
\[
x^{k+1}=P[\ixv^{k}](x^{k})\qquad\text{for all \ensuremath{k\geq1}.}
\]
When $r\geq1$, we have that $w^{r}(\ixv)>0$ just if $\ixv=\ixv^{r}$.
For this $\ixv$, we have that $w^{r}(\ixv)=1$ and when $r\geq i$
we have that $\ixv\in\Ivecs\langle i\rangle$, and $\Position(i,\ixv)\leq r$.
Hence $\sum_{\ixv\in\Ivecs\langle i\rangle}\frac{w^{r}(\ixv)}{\Position(i,\ixv)}\geq1/r$
if $r\geq i$. So the hypotheses of Corollary~\ref{cor-2ndconv}
are satisfied for $f_{r}=1/r$, and $(x^{k})_{k=0}^{\infty}$ converges
to a point in $\bigcap\mathcal{C}$.

\subsubsection*{Example 3}

For odd iteration numbers $k\geq1$, let $\ixv^{k}$ be some permutation
of $(1,3,5,\dots,k)$. For even iteration numbers $k\geq2$, let $\ixv^{k}$
be some permutation of $(2,4,6,\dots,k)$. Suppose 
\[
x^{k+1}=\frac{1}{2}P_{k+2}(x^{k})+\frac{1}{2}P[\ixv^{k}](x^{k})\quad\text{for all \ensuremath{k\geq1.}}
\]
Let us verify that the hypotheses of the Second Convergence Theorem
are satisfied when we put $\kappa_{r}=2r$ and $f_{r}=\frac{1}{2(r+1)}$
for all $r\geq1$.

Now \textbf{C1} holds because: 
\[
\sum_{k=\kappa_{r}}^{\kappa_{r+1}-2}\maxLength(w^{k})=\sum_{k=2r}^{2r}\maxLength(w^{k})=\maxLength(w^{2r})=\length(\ixv^{2r})=r<\frac{1}{f_{r}}
\]
It remains to verify \textbf{C2}.

As $\kappa_{r}=2r$ and $f_{r}=\frac{1}{2(r+1)}$, \textbf{C2} is
equivalent to the statement that for every index $j$ there exists
an $r_{j}\in\Z^{+}$ such that at least one of the following is true
for all $r\geq r_{j}$: 
\[
{\displaystyle {\sum_{\ixv\in\Ivecs\langle j\rangle}\frac{w^{2r}(\ixv)}{\Position(j,\ixv)}\geq\frac{1}{2(r+1)}\text{\quad or\quad}\sum_{\ixv\in\Ivecs\langle j\rangle}\frac{w^{2r+1}(\ixv)}{\Position(j,\ixv)}\geq\frac{1}{2(r+1)}}}
\]
When $j$ is \textit{even} this holds because if $\ixv=\ixv^{2r}$
and $2r\geq j$, then $\ixv\in\Ivecs\langle j\rangle$, $w^{2r}(\ixv)=\frac{1}{2}$,
and $\Position(j,\ixv)\leq r$, whence $\frac{w^{2r}(\ixv)}{\Position(j,\ixv)}>\frac{1}{2(r+1)}$.
When $j$ is \textit{odd} it holds because if $\ixv=\ixv^{2r+1}$
and $2r+1\geq j$, then $\ixv\in\Ivecs\langle j\rangle$, $w^{2r+1}(\ixv)=\frac{1}{2}$,
and $\Position(j,\ixv)\leq r+1$, whence $\frac{w^{2r+1}(\ixv)}{\Position(j,\ixv)}\geq\frac{1}{2(r+1)}$.
Therefore \textbf{C2} holds.

So $(x^{k})_{k=0}^{\infty}$ converges to a point in $\bigcap\mathcal{C}$.

\subsection{Proof of the Second Convergence Theorem}

\label{proof}

As a consequence of Lemma~\ref{lem-fact-one} (which implies $\|x^{k+1}-c\|\leq\|x^{k}-c\|$
for any $c\in\bigcap\mathcal{C}$ and all iteration numbers $k$),
we deduce:

\begin{lemma} \label{first-claim} If $x^{0},x^{1},x^{2},\dots$
has a subsequence that converges to a point $c$ in $\bigcap\mathcal{C}$,
then $x^{0},x^{1},x^{2},\dots$ converges to $c$. \hfill{}\qedsymbol
\end{lemma}

Our proof of the Second Convergence Theorem will be based on the following
claim:

\begin{claim} \label{third-claim} To show that the sequence of iterates
$x^{0},x^{1},x^{2},\dots$ converges to a point in $\bigcap\mathcal{C}$,
it is enough to show that for every index $q$ there is an integer
$k_{q}$ for which the following is true: 
\begin{equation}
\max_{1\leq i\leq q}d(x^{k_{q}},C_{i})\leq1/q\label{claim-ineq}
\end{equation}
\end{claim}
\begin{proof}
[Justification of Claim \ref{third-claim}]Suppose that for every
index $q$ there is an integer $k_{q}$ for which (\ref{claim-ineq})
is true. As $x^{k_{1}},x^{k_{2}},x^{k_{3}},\dots$ is a bounded sequence
in $\R^{n}$ (by Lemma~\ref{lem-fact-one}), it has a convergent
subsequence. For every index $q$ we see from (\ref{claim-ineq})
that $d(x^{k_{r}},C_{q})\leq1/r$ for all integers $r\geq q$, which
implies that if $x$ is the limit of a convergent subsequence of $x^{k_{1}},x^{k_{2}},x^{k_{3}},\dots$
then $d(x,C_{q})\leq1/r$ for all positive integers $r$ (i.e., $d(x,C_{q})=0$),
whence $x\in C_{q}$ (since $C_{q}$ is closed). Thus the limit of
a convergent subsequence of $x^{k_{1}},x^{k_{2}},x^{k_{3}},\dots$
is a point in $\bigcap\mathcal{C}$, and so it follows from Lemma~\ref{first-claim}
that $x^{0},x^{1},x^{2},\dots$ converges to that same point. 
\end{proof}
To show there is a sequence of iteration numbers $k_{1}<k_{2}<k_{3}<\dots$
such that (\ref{claim-ineq}) holds for every index $q$, we use the
inequalities stated in the following lemma and its corollary:

\begin{titulemma*} Let $m$ be any integer greater than or equal
to $2$, $z_{1},\dots,z_{m}$ any sequence of $m$ real numbers, and
$\lambda_{1},\dots,\lambda_{m}$ any sequence of $m$ positive real
numbers. Then: 
\[
\frac{z_{1}^{2}}{\lambda_{1}}+\dots+\frac{z_{m}^{2}}{\lambda_{m}}\geq\frac{(z_{1}+\dots+z_{m})^{2}}{\lambda_{1}+\dots+\lambda_{m}}
\]
\end{titulemma*}

\begin{proof} This well known inequality is equivalent to the Cauchy-Schwarz
inequality for the pair of vectors $(z_{1}/\sqrt{\lambda_{1}},\dots,z_{m}/\sqrt{\lambda_{m}})$
and $(\sqrt{\lambda_{1}},\dots,\sqrt{\lambda_{m}})$. \end{proof}


\begin{corollary*} Let $m$ be any integer greater than or equal
to 2, and $z_{1},\dots,z_{m}$ any sequence of $m$ real numbers.
Then we have that $z_{1}^{2}+\dots+z_{m}^{2}\geq(z_{1}+\dots+z_{m})^{2}/m$.
\hfill{}\qedsymbol \end{corollary*}

We now proceed to show the existence of iteration numbers $k_{1}<k_{2}<k_{3}<\dots$
that satisfy (\ref{claim-ineq}) for every index $q$.

\begin{lemma} \label{lem-new} Let $c$ be any point in $\bigcap\mathcal{C}$.
Let $y\in\R^{n}$, and let $\ixv=(\ixv_{1},\ixv_{2},\dots,\ixv_{m})$
be any index vector. Let $y_{0}=y$, and, for $1\leq j\leq m$, let
$y_{j}=P[(\ixv_{1},\ixv_{2},\ldots,\ixv_{j})](y)$ and let $d_{j}=\|y_{j}-y_{j-1}\|=d(y_{j-1},C_{\ixv_{j}})$.
Then for $1\leq j\leq m$ we have that: 
\begin{gather}
\|y_{j}-c\|^{2}\leq\|y_{j-1}-c\|^{2}-d_{j}^{2}\label{*}\\
\|y_{j}-c\|^{2}\leq\|y-c\|^{2}-(d_{1}^{2}+\dots+d_{j}^{2})\label{**}\\
\|y_{j}-y\|\leq d_{1}+\dots+d_{j}\label{**H}\\
\|y_{j}-c\|^{2}\leq\|y-c\|^{2}-\|y_{j}-y\|^{2}/j\label{***}\\
\|P[\ixv](y)-c\|^{2}=\|y_{m}-c\|^{2}\leq\|y_{j}-c\|^{2}\leq\|y-c\|^{2}-d(y,C_{\ixv_{j}})^{2}/j\label{eq:old-lem-3.6}
\end{gather}
\end{lemma}

\begin{proof} Since $y_{j}=P_{\ixv_{j}}(y_{j-1})$, (\ref{*}) follows
from Proposition~\ref{fact-one}; (\ref{**}) follows from (\ref{*}).
Inequality (\ref{**H}) follows from the triangle inequality. Inequality
(\ref{***}) follows from (\ref{**}), the corollary to Titu's Lemma,
and (\ref{**H}). The first inequality of (\ref{eq:old-lem-3.6})
follows from the fact that (\ref{*}) holds for $1\leq j\leq m$.
The second inequality of (\ref{eq:old-lem-3.6}) follows from (\ref{***})
and the fact that $\|y_{j}-y\|\geq d(y,C_{\ixv_{j}})$ because $y_{j}\in C_{\ixv_{j}}$.
\end{proof}

\begin{lemma} \label{lem:eq-ineq1} Let $c\in\bigcap\mathcal{C}$,
let $i$ be any index, and let $k$ be any iteration number. Then:
\[
\|x^{k+1}-c\|^{2}\leq\|x^{k}-c\|^{2}-d(x^{k},C_{i})^{2}\sum_{\ixv\in\Ivecs\langle i\rangle}\frac{w^{k}(\ixv)}{\Position(i,\ixv)}
\]
 \end{lemma}

\begin{proof} On putting $j=\Position(i,\ixv)$ in (\ref{eq:old-lem-3.6}),
we see that if $\ixv\in\Ivecs\langle i\rangle$ and $y\in\R^{n}$
then: 
\begin{equation}
\|P[\ixv](y)-c\|^{2}\leq\|y-c\|^{2}-\frac{d(y,C_{i})^{2}}{\Position(i,\ixv)}\label{eq:new}
\end{equation}
Moreover, for any $y\in\R^{n}$ we see from (\ref{eq:old-lem-3.6})
that $\|P[\ixv](y)-c\|^{2}\leq\|y-c\|^{2}$ even if $\ixv\not\in\Ivecs\langle i\rangle$.
From this, (\ref{eq:new}), and (\ref{fin-alg}) we deduce that: 
\begin{align*}
\|x^{k+1}-c\|^{2} & \,\,\leq\,\,{\textstyle \sum_{\ixv\in\Ivecs\langle i\rangle}w^{k}(\ixv)\|P[\ixv](x^{k})-c\|^{2}\quad+\quad\sum_{\ixv\not\in\Ivecs\langle i\rangle}w^{k}(\ixv)\|P[\ixv](x^{k})-c\|^{2}}\\
 & \,\,\leq\,\,{\textstyle \sum_{\ixv\in\Ivecs\langle i\rangle}w^{k}(\ixv)\left(\|x^{k}-c\|^{2}-\frac{d(x^{k},C_{i})^{2}}{\Position(i,\ixv)}\right)\quad+\quad\sum_{\ixv\not\in\Ivecs\langle i\rangle}w^{k}(\ixv)\|x^{k}-c\|^{2}}
\end{align*}
Here the right side is equal to $\|x^{k}-c\|^{2}-d(x^{k},C_{i})^{2}\sum_{\ixv\in\Ivecs\langle i\rangle}\frac{w^{k}(\ixv)}{\Position(i,\ixv)}$
because $\sum_{\ixv\in\Ivecs}w^{k}(\ixv)=1$. \end{proof}

\begin{lemma} \label{lem:fund inq2} Let $c\in\bigcap\mathcal{C}$
and let $k$ be any iteration number. Then we have that: 
\begin{equation}
\|x^{k+1}-c\|^{2}\leq\|x^{k}-c\|^{2}-\frac{\|x^{k+1}-x^{k}\|^{2}}{\maxLength(w^{k})}\label{fund inq2}
\end{equation}
\end{lemma}
\begin{proof}
For each index vector $\ixv$ used at iteration $k$, $\length(\ixv)\leq\maxLength(w^{k})$
and so (\ref{***}) implies $\|P[\ixv](x^{k})-c\|^{2}\leq\|x^{k}-c\|^{2}-\|P[\ixv](x^{k})-x^{k}\|^{2}/\maxLength(w^{k})$
or, equivalently, 
\begin{equation}
\|P[\ixv](x^{k})-c\|^{2}+\|P[\ixv](x^{k})-x^{k}\|^{2}/\maxLength(w^{k})\leq\|x^{k}-c\|^{2}\label{fund inq}
\end{equation}
As $u\mapsto\|u-c\|^{2}$ and $u\mapsto\|u-x^{k}\|^{2}$ are convex
functions on $\R^{n}$, we have that: 
\begin{alignat*}{2}
\|x^{k+1}-c\|^{2} & +\frac{\|x^{k+1}-x^{k}\|^{2}}{\maxLength(w^{k})}\\
 & \ =\ {\displaystyle {\left\Vert \left({\textstyle \sum_{\ixv\in\Ivecs}w^{k}(\ixv)P[\ixv](x^{k})}\right)-c\right\Vert ^{2}+\frac{\left\Vert \left(\sum_{\ixv\in\Ivecs}w^{k}(\ixv)P[\ixv](x^{k})\right)-x^{k}\right\Vert ^{2}}{\maxLength(w^{k})}}}\\
 & \ \leq\ {\displaystyle {{\textstyle \sum_{\ixv\in\Ivecs}w^{k}(\ixv)\|P[\ixv](x^{k})-c\|^{2}}+\frac{\sum_{\ixv\in\Ivecs}w^{k}(\ixv)\|P[\ixv](x^{k})-x^{k}\|^{2}}{\maxLength(w^{k})}}\qquad\quad}\,\,\,\, & \text{by (\ref{Jensen})}\\
 & \ =\ {\textstyle {\sum_{\ixv\in\Ivecs}w^{k}(\ixv)\left(\|P[\ixv](x^{k})-c\|^{2}+\|P[\ixv](x^{k})-x^{k}\|^{2}/\maxLength(w^{k})\right)}\notag}\\
 & \ \leq\ {\textstyle {\sum_{\ixv\in\Ivecs}w^{k}(\ixv)\|x^{k}-c\|^{2}}} & \text{by (\ref{fund inq})}\\
 & \ =\ \|x^{k}-c\|^{2} & \text{as \ensuremath{{\textstyle \sum_{\ixv\in\Ivecs}w^{k}(\ixv)=1}}}
\end{alignat*}
which immediately implies (\ref{fund inq2}). 
\end{proof}
\begin{corollary} \label{cor *} Let $\kappa<\kappa^{*}$ be iteration
numbers and let $c\in\bigcap\mathcal{C}$. Then: 
\begin{equation}
\|x^{\kappa^{*}}-c\|^{2}\leq\|x^{\kappa}-c\|^{2}-\frac{\|x^{\kappa^{*}}-x^{\kappa}\|^{2}}{\sum_{k=\kappa}^{\kappa^{*}-1}\maxLength(w^{k})}\label{eq:cor *}
\end{equation}
\end{corollary}

\begin{proof} We have that 
\begin{alignat*}{2}
\|x^{\kappa^{*}}-c\|^{2} & \ \leq\ \|x^{\kappa}-c\|^{2}-{\textstyle \sum_{k=\kappa}^{\kappa^{*}-1}}\frac{\|x^{k+1}-x^{k}\|^{2}}{\maxLength(w^{k})} &  & \qquad\text{by Lemma~\ref{lem:fund inq2}}\\
 & \ \leq\ \|x^{\kappa}-c\|^{2}-\frac{\left(\sum_{k=\kappa}^{\kappa^{*}-1}\|x^{k+1}-x^{k}\|\right)^{2}}{\sum_{k=\kappa}^{\kappa^{*}-1}\maxLength(w^{k})} &  & \qquad\text{by Titu's Lemma}\\
 & \ \leq\ \|x^{\kappa}-c\|^{2}-\frac{\|x^{\kappa^{*}}-x^{\kappa}\|^{2}}{\sum_{k=\kappa}^{\kappa^{*}-1}\maxLength(w^{k})} &  & \qquad\text{by the triangle inequality}
\end{alignat*}
which establishes (\ref{eq:cor *}). \end{proof}

\begin{lemma} \label{lem main2} Let $c\in\bigcap\mathcal{C}$, let
$\kappa\leq\kappa^{*}<\kappa'$ be iteration numbers, and let $i$
be any index. Then 
\[
\|x^{\kappa'}-c\|^{2}\leq\|x^{\kappa}-c\|^{2}-\min\left({\displaystyle \frac{1}{\sum_{k=\kappa}^{\kappa^{*}-1}\maxLength(w^{k})}},{\displaystyle \sum_{\ixv\in\Ivecs\langle i\rangle}\frac{w^{\kappa^{*}}(\ixv)}{\Position(i,\ixv)}}\right)\frac{d(x^{\kappa},C_{i})^{2}}{4},
\]
where $\min\left(\frac{1}{\sum_{k=\kappa}^{\kappa^{*}-1}\maxLength(w^{k})},\sum_{\ixv\in\Ivecs\langle i\rangle}\frac{w^{\kappa^{*}}(\ixv)}{\Position(i,\ixv)}\right)$
should be understood to mean $\sum_{\ixv\in\Ivecs\langle i\rangle}\frac{w^{\kappa^{*}}(\ixv)}{\Position(i,\ixv)}$
if $\kappa=\kappa^{*}$. \end{lemma}

\begin{proof} We first consider the case in which $\|x^{\kappa^{*}}-x^{\kappa}\|>d(x^{\kappa},C_{i})/2$.
In this case $\kappa<\kappa^{*}<\kappa'$ and so 
\begin{alignat*}{2}
\|x^{\kappa'}-c\|^{2} & \leq\|x^{\kappa^{*}}-c\|^{2} & \qquad & \text{by Lemma~\ref{lem-fact-one}}\\
 & \leq\|x^{\kappa}-c\|^{2}-\frac{\|x^{\kappa^{*}}-x^{\kappa}\|^{2}}{\sum_{k=\kappa}^{\kappa^{*}-1}\maxLength(w^{k})} & \qquad & \text{by Corollary~\ref{cor *}}\\
 & <\|x^{\kappa}-c\|^{2}-\left(\frac{1}{\sum_{k=\kappa}^{\kappa^{*}-1}\maxLength(w^{k})}\right)\frac{d(x^{\kappa},C_{i})^{2}}{4} & \qquad & \text{as \ensuremath{\|x^{\kappa^{*}}-x^{\kappa}\|>d(x^{\kappa},C_{i})/2}}
\end{alignat*}
and the lemma holds.

Now we consider the case in which $\|x^{\kappa^{*}}-x^{\kappa}\|\leq d(x^{\kappa},C_{i})/2$.
In this case we see from the triangle inequality that $d(x^{\kappa^{*}},C_{i})\geq d(x^{\kappa},C_{i})/2$,
whence 
\begin{equation}
d(x^{\kappa^{*}},C_{i})^{2}\ \geq\ \frac{d(x^{\kappa},C_{i})^{2}}{4}\label{eq:newbound}
\end{equation}
Therefore, since $\kappa\leq\kappa^{*}<\kappa'$, we have that 
\begin{alignat*}{2}
\|x^{\kappa'}-c\|^{2} & \leq\|x^{\kappa^{*}+1}-c\|^{2} & \qquad & \text{by Lemma~\ref{lem-fact-one}}\\
 & \leq\|x^{\kappa^{*}}-c\|^{2}-d(x^{\kappa^{*}},C_{i})^{2}\sum_{\ixv\in\Ivecs\langle i\rangle}\frac{w^{\kappa^{*}}(\ixv)}{\Position(i,\ixv)} & \qquad & \text{by Lemma~\ref{lem:eq-ineq1}}\\
 & \leq\|x^{\kappa^{*}}-c\|^{2}-\frac{d(x^{\kappa},C_{i})^{2}}{4}\sum_{\ixv\in\Ivecs\langle i\rangle}\frac{w^{\kappa^{*}}(\ixv)}{\Position(i,\ixv)} & \qquad & \text{by (\ref{eq:newbound})}\\
 & \leq\|x^{\kappa}-c\|^{2}-\frac{d(x^{\kappa},C_{i})^{2}}{4}\sum_{\ixv\in\Ivecs\langle i\rangle}\frac{w^{\kappa^{*}}(\ixv)}{\Position(i,\ixv)} &  & \text{by Lemma~\ref{lem-fact-one}}
\end{alignat*}
and so the lemma holds in this case too. \end{proof}

\subsubsection*{Completion of the Proof of the Second Convergence Theorem}

Suppose the hypotheses of the theorem are satisfied, so that $f_{1},f_{2},f_{3},\dots\in[0,1]$
satisfy $\sum_{r=1}^{\infty}f_{r}=\infty$ and $\kappa_{1}<\kappa_{2}<\kappa_{3}<\dots$
are iteration numbers such that 
\begin{lyxlist}{00.000}
\item [{\hspace{2.25em}\textbf{C1}:}] $\sum_{k=\kappa_{r}}^{\kappa_{r+1}-2}\maxLength(w^{k})\leq1/f_{r}$
for all $r\in\Z^{+}$ such that $\kappa_{r+1}\geq\kappa_{r}+2$ and
$f_{r}>0$. 
\item [{\hspace{2.25em}\textbf{C2}:}] For every index $i$ such that $C_{i}\neq\R^{n}$,
${\displaystyle {\max_{\kappa_{r}\leq k<\kappa_{r+1}}\sum_{\ixv\in\Ivecs\langle i\rangle}\frac{w^{k}(\ixv)}{\Position(i,\ixv)}}\geq f_{r}}$
for all sufficiently large $r\in\Z^{+}$. 
\end{lyxlist}
For every index $i$ such that $C_{i}\neq\R^{n}$ and all $r\in\Z^{+}$,
let $\kappa_{r}^{*}(i)$ be an iteration number such that $\kappa_{r}\leq\kappa_{r}^{*}(i)<\kappa_{r+1}$
and $\kappa_{r}^{*}(i)\in{\displaystyle {\argmax_{\kappa_{r}\leq k<\kappa_{r+1}}\sum_{\ixv\in\Ivecs\langle i\rangle}\frac{w^{k}(\ixv)}{\Position(i,\ixv)}}}$.
{[}Here ${\displaystyle {\argmax_{\kappa_{r}\leq k<\kappa_{r+1}}\sum_{\ixv\in\Ivecs\langle i\rangle}\frac{w^{k}(\ixv)}{\Position(i,\ixv)}}}$
denotes the set of those iteration numbers $k$ that maximize $\sum_{\ixv\in\Ivecs\langle i\rangle}\frac{w^{k}(\ixv)}{\Position(i,\ixv)}$
subject to the condition $\kappa_{r}\leq k<\kappa_{r+1}$.{]}

For every index $i$ such that $C_{i}\neq\R^{n}$, $\kappa_{r}^{*}(i)<\kappa_{r+1}$
implies
\[
\sum_{k=\kappa_{r}}^{\kappa_{r}^{*}(i)-1}\maxLength(w^{k})\leq\sum_{k=\kappa_{r}}^{\kappa_{r+1}-2}\maxLength(w^{k})
\]
 and so we see from \textbf{C1} and \textbf{C2} that 
\begin{equation}
\min\left(\frac{1}{\sum_{k=\kappa_{r}}^{\kappa_{r}^{*}(i)-1}\maxLength(w^{k})},\sum_{\ixv\in\Ivecs\langle i\rangle}\frac{w^{\kappa_{r}^{*}(i)}(\ixv)}{\Position(i,\ixv)}\right)\geq f_{r}\label{eq:2ndnewbound}
\end{equation}
if $r$ is sufficiently large, where the left side should be understood
to mean $\sum_{\ixv\in\Ivecs\langle i\rangle}\frac{w^{\kappa_{r}^{*}(i)}(\ixv)}{\Position(i,\ixv)}$
if $\kappa_{r}=\kappa_{r}^{*}(i)$.

For every index $i$ such that $C_{i}\neq\R^{n}$, it follows from
Lemma~\ref{lem main2} (on putting $(\kappa,\kappa^{*},\kappa')=(\kappa_{r},\kappa_{r}^{*}(i),\kappa_{r+1})$),
and the fact that (\ref{eq:2ndnewbound}) holds for all sufficiently
large $r$, that 
\begin{equation}
\|x^{\kappa_{r+1}}-c\|^{2}\leq\|x^{\kappa_{r}}-c\|^{2}-\frac{f_{r}}{4}d(x^{\kappa_{r}},C_{i})^{2}\label{eq:step}
\end{equation}
if $r$ is sufficiently large. Moreover, (\ref{eq:step}) certainly
holds for every $r\in\Z^{+}$ if $C_{i}=\R^{n}$ (so that $d(x^{\kappa_{r}},C_{i})=0$),
by Lemma~\ref{lem-fact-one}. It follows that, for every index $q$,
\begin{equation}
\|x^{\kappa_{r+1}}-c\|^{2}\leq\|x^{\kappa_{r}}-c\|^{2}-\frac{f_{r}}{4}\max_{1\leq i\leq q}d(x^{\kappa_{r}},C_{i})^{2}\label{eq:3rdnewbound}
\end{equation}
if $r$ is sufficiently large.

We claim that for any index $q$ there is a positive integer $k_{q}$
such that: 
\begin{equation}
\max_{1\leq i\leq q}d(x^{k_{q}},C_{i})\leq1/q\label{final-claim-ineq}
\end{equation}
Indeed, if for some index $q$ no such $k_{q}$ existed then it would
follow from (\ref{eq:3rdnewbound}) that 
\[
\|x^{\kappa_{r+1}}-c\|^{2}\leq\|x^{\kappa_{r}}-c\|^{2}-\frac{f_{r}}{4q^{2}}
\]
holds for all sufficiently large $r$, which is impossible as $\sum_{r=0}^{\infty}f_{r}$
diverges.

The existence for each index $q$ of some $k_{q}$ that satisfies
(\ref{final-claim-ineq}) is enough to guarantee convergence of $x^{0},x^{1},x^{2},\dots$
to a point in $\bigcap\mathcal{C}$, as we observed in Claim~\ref{third-claim}.
\hfill{}\qedsymbol

\section{Third Convergence Theorem}

\label{finite convergence}

We continue to assume that $x^{0},x^{1},x^{2},x^{3},\dots$ is an
arbitrary sequence of points in $\R^{n}$ that satisfy (\ref{alg})
for some weight functions $w^{0},w^{1},w^{2},w^{3},\dots:\Ivecs\rightarrow[0,1]$.
The convergence theorem of this section is similar in form to the
First Convergence Theorem. The difference is that, instead of assuming
$\bigcap\mathcal{C}$ is an $n$-dimensional subset of $\R^{n}$,
it assumes $\textbf{UASC}(\mathcal{C})$ holds:

\begin{theorem}[Third Convergence Theorem] \label{finite-000} Suppose
\emph{$\textbf{UASC}(\mathcal{C})$} holds. Then the sequence $x^{0},x^{1},x^{2},x^{3},\dots$
converges, and its limit lies in $C_{j}$ for every index $j$ such
that $\sum_{k=0}^{\infty}\sum_{\ixv\in\Ivecs\langle j\rangle}w^{k}(\ixv)=\infty$.
In particular, the sequence $x^{0},x^{1},x^{2},\dots$ converges to
a point in $\bigcap\mathcal{C}$ if $\sum_{k=0}^{\infty}\sum_{\ixv\in\Ivecs\langle j\rangle}w^{k}(\ixv)=\infty$
for every index $j$ such that $C_{j}\neq\R^{n}$.\end{theorem}

This theorem will be proved in the rest of this section.

\subsection{Useful Lemmas}

\label{sec-useful-lemmas}

Our proof of the Third Convergence Theorems will depend on the lemmas
in this subsection. We note that these lemmas are true regardless
of whether or not $\textbf{UASC}(\mathcal{C})$ holds.

For any point $s\in\R^{n}$, let $\IxSetContaining{s}$ and $\IxSetNotContaining{s}$
be the complementary sets of indices and $\IvecsAllInvolving{s}$
and $\IvecsNotInvolving{s}$ the complementary sets of index vectors
that are defined as follows: 
\begin{align*}
\IxSetContaining{s} & \ =\ \{i\in\Z^{+}\mid s\in C_{i}\}\\
\IxSetNotContaining{s} & \ =\ \{i\in\Z^{+}\mid s\not\in C_{i}\}\ =\ \Z^{+}\setminus\IxSetContaining{s}\\
\IvecsAllInvolving{s} & \ =\ \{\ixv\in\Ivecs\mid\IxSetOfVec(\ixv)\subseteq\IxSetContaining{s}\}\\
\IvecsNotInvolving{s} & \ =\ \{\ixv\in\Ivecs\mid\IxSetOfVec(\ixv)\cap\IxSetNotContaining{s}\neq\emptyset\}\ =\ {\textstyle {\bigcup_{i\in\IxSetNotContaining{s}}\Ivecs\langle i\rangle}\ =\ \Ivecs\setminus\IvecsAllInvolving{s}\notag}
\end{align*}
If $s\notin\bigcap\mathcal{C}$, then $\IxSetNotContaining{s}$ and
$\IvecsNotInvolving{s}$ are nonempty; if $s\in\bigcap\mathcal{C}$,
then $\IxSetNotContaining{s}$ and $\IvecsNotInvolving{s}$ are empty.

As $\ixv\in\IvecsAllInvolving{s}$ if and only if $s\in\bigcap_{i\in\IxSetOfVec(\ixv)}C_{i}$,
assertion 1 of Lemma~\ref{gen-fact-zero} implies:

\begin{lemma} \label{lem-two} Let $s,z\in\R^{n}$ and let $\ixv\in\IvecsAllInvolving{s}$.
Then $\|P[\ixv](z)-s\|\leq\|z-s\|$. \hfill{}\qedsymbol \end{lemma}

The number $2\mspace{1mu}d(x^{0},\bigcap\mathcal{C})$ will be denoted
by $M$. Our next lemma gives an upper bound (in terms of $M$) on
how fast the iterates $x^{k}$ can move away from any given point
$s\in\R^{n}$ as the iteration number $k$ increases. Note that $M$
does not depend on the iteration numbers $r_{1}$ and $r_{2}$ of
this lemma.

\begin{lemma} \label{finite-third} Let $r_{1}$ and $r_{2}$ be
iteration numbers such that $r_{1}<r_{2}$, and let $s\in\R^{n}$.
Then we have that $\|x^{r_{2}}-s\|\leq\|x^{r_{1}}-s\|+M\sum_{k=r_{1}}^{r_{2}-1}{\textstyle {\sum_{\ixv\in\IvecsNotInvolving{s}}w^{k}(\ixv)}}$.
\end{lemma}

\begin{proof} For all index vectors $\ixv$, all iteration numbers
$k$, and every $c\in\bigcap\mathcal{C}$, we see from Lemmas~\ref{gen-fact-zero}
and \ref{lem-fact-one} that $\|P[\ixv](x^{k})-c\|\leq\|x^{k}-c\|\leq\|x^{0}-c\|$,
which (by the triangle inequality) implies $\|P[\ixv](x^{k})-x^{k}\|\leq2\|x^{0}-c\|$.
Since this is true for every point $c$ in $\bigcap\mathcal{C}$,
we have that 
\[
{\textstyle \|P[\ixv](x^{k})-x^{k}\|\leq2\mspace{1mu}d(x^{0},\bigcap\mathcal{C})=M\qquad\text{for all \ensuremath{k\geq0}}}
\]
and hence that 
\begin{equation}
\|P[\ixv](x^{k})-s\|\leq\|x^{k}-s\|+\|P[\ixv](x^{k})-x^{k}\|\leq\|x^{k}-s\|+M\qquad\text{for all \ensuremath{k\geq0}.}\label{finite-eq00}
\end{equation}
Moreover, if $\ixv\in\IvecsAllInvolving{s}$ then Lemma~\ref{lem-two}
implies that 
\begin{equation}
\|P[\ixv](x^{k})-s\|\leq\|x^{k}-s\|\qquad\text{for all \ensuremath{k\geq0}}.\label{finite-eq000}
\end{equation}
Hence 
\begin{alignat*}{2}
\|x^{k+1}-s\| & \ =\ {\textstyle {\left\Vert \left(\sum_{\ixv\in\Ivecs}w^{k}(\ixv)P[\ixv](x^{k})\right)-s\right\Vert }}\\
 & \ \leq\ {\textstyle {\sum_{\ixv\in\Ivecs}w^{k}(\ixv)\|P[\ixv](x^{k})-s\|}} &  & \text{\qquad by convexity of \ensuremath{u\mapsto\|u-s\|}}\\
 & \ =\ {\textstyle {\sum_{\ixv\in\IvecsAllInvolving{s}}w^{k}(\ixv)\|P[\ixv](x^{k})-s\|}}\\
 & \qquad\qquad+{\textstyle {\sum_{\ixv\in\IvecsNotInvolving{s}}w^{k}(\ixv)\|P[\ixv](x^{k})-s\|}}\\
 & \ \leq\ {\textstyle {\sum_{\ixv\in\IvecsAllInvolving{s}}w^{k}(\ixv)\|x^{k}-s\|}}\\
 & \qquad\qquad+{\textstyle {\sum_{\ixv\in\IvecsNotInvolving{s}}w^{k}(\ixv)(\|x^{k}-s\|+M)}} &  & \text{\qquad by (\ref{finite-eq000}) and (\ref{finite-eq00})}\\
 & \ =\ \|x^{k}-s\|\;+\;{\textstyle {M\sum_{\ixv\in\IvecsNotInvolving{s}}w^{k}(\ixv)}} &  & \qquad\text{as \ensuremath{{\textstyle {\sum_{\ixv\in\Ivecs}w^{k}(\ixv)=1}}}}
\end{alignat*}
for all iteration numbers $k$. The lemma follows from this. \end{proof}

\begin{corollary} \label{cor-converge} Let $s$ be the limit of
a convergent subsequence of $x^{0},x^{1},x^{2},\dots$, and suppose
the series $\sum_{k=0}^{\infty}{\textstyle {\sum_{\ixv\in\IvecsNotInvolving{s}}w^{k}(\ixv)}}$
converges. Then the sequence $x^{0},x^{1},x^{2},\dots$ converges
to $s$. \end{corollary}

\begin{proof} For all $\epsilon>0$, let $N(\epsilon)$ be an iteration
number such that $M\sum_{r=N(\epsilon)}^{\infty}{\sum_{\ixv\in\IvecsNotInvolving{s}}w^{r}(\ixv)}<\epsilon/2$,
and let $\kappa(\epsilon)$ be an iteration number such that $\kappa(\epsilon)\geq N(\epsilon)$
and $\|x^{\kappa(\epsilon)}-s\|<\epsilon/2$. Then, for all $\epsilon>0$,
we see from Lemma~\ref{finite-third} that $\|x^{k}-s\|<\epsilon$
when $k\geq\kappa(\epsilon)$. \end{proof}

For all $s\in\R^{n}\setminus\bigcap\mathcal{C}$, let us now define:
\[
{\textstyle {d(s)=d(s,\bigcup_{i\in\IxSetNotContaining{s}}C_{i})}}
\]

\begin{lemma}\label{lem-three} Let $c\in\bigcap\mathcal{C}$, let
$s\in\R^{n}\setminus\bigcap\mathcal{C}$, let $\ixv\in\IvecsNotInvolving{s}$,
and let $z$ be a point in $\R^{n}$ such that $\|z-s\|\leq d(s)/2$.
Then we have that $\|P[\ixv](z)-c\|^{2}\leq\|z-c\|^{2}-d(s)^{2}/4$.\end{lemma}

\begin{proof} Let $\ixv=(\ixv_{1},\ldots,\ixv_{m})$, so that $P[\ixv]=P_{\ixv_{m}}\circ\ldots\circ P_{\ixv_{1}}$.
Let $\ixv_{j}$ be the \emph{first} one of the components $\ixv_{1},\ldots,\ixv_{m}$
for which $s\not\in C_{\ixv_{j}}$; $\ixv_{j}$ exists, since $\ixv\in\IvecsNotInvolving{s}$.
Let $\ixv'=(\ixv_{1},\ldots,\ixv_{j-1})$ and let $\ixv''=(\ixv_{j+1},\ldots,\ixv_{m})$,
so that $P[\ixv]=P[\ixv'']\circ P_{\ixv_{j}}\circ P[\ixv']$. (If
$j=1$ or $j=m$, then $\ixv'$ or $\ixv''$ is the empty sequence
and we define $P[\ixv']$ or $P[\ixv'']$ to be the identity map.)
As $\ixv'\in\IvecsAllInvolving{s}$ if $j\neq1$, we see from Lemma~\ref{lem-two}
that $\|P[\ixv'](z)-s\|\leq\|z-s\|\leq d(s)/2$. So, since $d(s)=d(s,\bigcup_{i\in\IxSetNotContaining{s}}C_{i})$,
it follows from the triangle inequality that the point $P[\ixv'](z)$
satisfies: 
\begin{equation}
{\textstyle {d(P[\ixv'](z),\bigcup_{i\in\IxSetNotContaining{s}}C_{i})\geq d(s)-d(s)/2=d(s)/2}}\label{eq-zero}
\end{equation}
\nop From this it follows that 
\begin{alignat*}{2}
\|P[\ixv](z)-c\|^{2} & \quad=\quad\|P[\ixv''](P_{\ixv_{j}}(P[\ixv'](z)))-c\|^{2}\\
 & \quad\leq\quad\|P_{\ixv_{j}}(P[\ixv'](z))-c\|^{2} &  & \qquad\text{by Lemma~\ref{gen-fact-zero}}\\
 & \quad\leq\quad\|P[\ixv'](z)-c\|^{2}-d(P[\ixv'](z),C_{\ixv_{j}})^{2} &  & \qquad\text{by Proposition~\ref{fact-one}}\\
 & \quad\leq\quad\|z-c\|^{2}-d(P[\ixv'](z),C_{\ixv_{j}})^{2} &  & \qquad\text{by Lemma~\ref{gen-fact-zero}}\\
 & \quad\leq\quad\|z-c\|^{2}-d(P[\ixv'](z),{\textstyle {\bigcup_{i\in\IxSetNotContaining{s}}C_{i}})^{2}} &  & \qquad\text{since \ensuremath{\ixv_{j}\in\IxSetNotContaining{s}}}\\
 & \quad\leq\quad{\textstyle {\|z-c\|^{2}-d(s)^{2}/4}} &  & \qquad\text{by (\ref{eq-zero})}
\end{alignat*}
and so the proof is complete. \end{proof}

\begin{lemma} \label{lem-four} Let $c\in\bigcap\mathcal{C}$ and
let $s\in\R^{n}\setminus\bigcap\mathcal{C}$. Let $k$ be any iteration
number such that $\|x^{k}-s\|\leq d(s)/2$. Then we have that $\|x^{k+1}-c\|^{2}\leq\|x^{k}-c\|^{2}-d(s)^{2}{\textstyle {\sum_{\ixv\in\IvecsNotInvolving{s}}w^{k}(\ixv)}/4}$.\end{lemma}

\begin{proof} We have that 
\begin{alignat*}{2}
\|x^{k+1}-c\|^{2} & \ \leq\ {\textstyle {\sum_{\ixv\in\IvecsNotInvolving{s}}w^{k}(\ixv)\|P[\ixv](x^{k})-c\|^{2}}}\\
 & \qquad\quad+\quad{\textstyle {\sum_{\ixv\in\IvecsAllInvolving{s}}w^{k}(\ixv)\|P[\ixv](x^{k})-c\|^{2}}} & \quad\quad & \text{by (\ref{fin-alg})}\\
 & \ \leq\ {\textstyle {\sum_{\ixv\in\IvecsNotInvolving{s}}w^{k}(\ixv)\left(\|x^{k}-c\|^{2}-d(s)^{2}/4\right)}}\\
 & \qquad\quad+\quad{\textstyle {\sum_{\ixv\in\IvecsAllInvolving{s}}w^{k}(\ixv)\|x^{k}-c\|^{2}}} &  & \text{by Lemmas \ref{lem-three} and \ref{gen-fact-zero}}\\
 & \ =\ {\textstyle {\|x^{k}-c\|^{2}-\sum_{\ixv\in\IvecsNotInvolving{s}}w^{k}(\ixv)d(s)^{2}/4}} &  & \text{since \ensuremath{{\textstyle {\sum_{\ixv\in\Ivecs}w^{k}(\ixv)=1}}}}
\end{alignat*}
as required. \end{proof}

\begin{corollary} \label{finite-cor} Let $c\in\bigcap\mathcal{C}$,
let $s\in\R^{n}\setminus\bigcap\mathcal{C}$, and let $r_{1}<r_{2}$
be iteration numbers such that: 
\[
\|x^{k}-s\|\leq d(s)/2\quad\text{for each \ensuremath{k}\,in the range \ensuremath{r_{1}\leq k\leq r_{2}-1}}.
\]
Then we have that\quad{}$\|x^{r_{2}}-c\|^{2}\leq\|x^{r_{1}}-c\|^{2}-d(s)^{2}\sum_{k=r_{1}}^{r_{2}-1}{\textstyle {\sum_{\ixv\in\IvecsNotInvolving{s}}w^{k}(\ixv)}/4}$.
\hfill{}\qedsymbol \end{corollary}

\nop From this corollary and Lemma~\ref{finite-third} we now deduce:

\begin{lemma} \label{lem-fund} Suppose $x_{0}\not\in\bigcap\mathcal{C}$,
so that $M>0$. Let $c\in\bigcap\mathcal{C}$, let $s\in\R^{n}\setminus\bigcap\mathcal{C}$,
and let $r_{1}$ and $r_{2}$ be iteration numbers such that $r_{1}<r_{2}$
for which the following conditions hold: 

\begin{lyxlist}{00.000}
\item [{\hspace{2.25em}1.}] \noindent $\|x^{r_{1}}-s\|\leq d(s)/4$. 
\item [{\hspace{2.25em}2.}] $\sum_{k=r_{1}}^{r_{2}-2}{\textstyle {\sum_{\ixv\in\IvecsNotInvolving{s}}w^{k}(\ixv)}\leq d(s)/4M}$. 
\end{lyxlist}
Then $\|x^{r_{2}}-c\|^{2}\leq\|x^{r_{1}}-c\|^{2}-d(s)^{2}\sum_{k=r_{1}}^{r_{2}-1}{\textstyle {\sum_{\ixv\in\IvecsNotInvolving{s}}w^{k}(\ixv)}/4}$.
\end{lemma}

\begin{proof} In view of Corollary~\ref{finite-cor}, it is enough
to show $\|x^{k}-s\|\leq d(s)/2$ for each $k$ in the range $r_{1}\leq k\leq r_{2}-1$.
This is indeed the case, because for any such $k$ we have that 
\begin{alignat*}{2}
\|x^{k}-s\| & \quad\leq\quad\|x^{r_{1}}-s\|+{\textstyle {M\sum_{j=r_{1}}^{k-1}\sum_{\ixv\in\IvecsNotInvolving{s}}w^{k}(\ixv)}} &  & \qquad\text{by Lemma~\ref{finite-third}}\\
 & \quad\leq\quad\|x^{r_{1}}-s\|+{\textstyle {M\sum_{j=r_{1}}^{r_{2}-2}\sum_{\ixv\in\IvecsNotInvolving{s}}w^{k}(\ixv)}\notag}\\
 & \quad\leq\quad{\textstyle {d(s)/4+d(s)/4}}\quad=\quad d(s)/2 &  & \qquad\text{by conditions 1 and 2}
\end{alignat*}
as required. \end{proof}

If $\sum_{k=0}^{\infty}{\textstyle {\sum_{\ixv\in\IvecsNotInvolving{s}}w^{k}(\ixv)}}$
diverges then, for any $r_{1}$ that satisfies condition 1 of Lemma~\ref{lem-fund},
how can we choose $r_{2}$ to get the best lower bound for $\|x^{r_{1}}-c\|^{2}-\|x^{r_{2}}-c\|^{2}$
from this lemma? The answer is that we should choose $r_{2}$ to be
the greatest value for which condition~2 of Lemma~\ref{lem-fund}
is satisfied. Most importantly, for that choice of $r_{2}$ we get
a lower bound for the difference $\|x^{r_{1}}-c\|^{2}-\|x^{r_{2}}-c\|^{2}$
which \emph{does not depend on $r_{1}$}. Thinking along these lines,
we deduce the following corollary of the lemma:

\begin{corollary} \label{cor-fund0} Suppose $x_{0}\not\in\bigcap\mathcal{C}$,
so that $M>0$. Let $c\in\bigcap\mathcal{C}$, and let $s\in\R^{n}\setminus\bigcap\mathcal{C}$
be such that the series $\sum_{k=0}^{\infty}{\textstyle {\sum_{\ixv\in\IvecsNotInvolving{s}}w^{k}(\ixv)}}$
diverges. Then for any iteration number $\kappa$ such that 
\begin{equation}
{\textstyle {\|x^{\kappa}-s\|\leq d(s)/4}}\label{eq:pre}
\end{equation}
the following inequality holds for all sufficiently large iteration
numbers $\kappa'$: 
\begin{equation}
{\textstyle {\|x^{\kappa'}-c\|^{2}\leq\|x^{\kappa}-c\|^{2}-d(s)^{3}/16M}}\label{eq:post}
\end{equation}
\end{corollary}

\begin{proof} Let $\kappa$ be an iteration number for which (\ref{eq:pre})
holds, let $r_{1}=\kappa$, and let $r_{2}$ be the iteration number
such that $r_{2}\geq r_{1}+1$ and both of the following are true:
\begin{align}
{\textstyle {\sum_{k=r_{1}}^{r_{2}-2}{\textstyle {\sum_{\ixv\in\IvecsNotInvolving{s}}w^{k}(\ixv)}}}} & \,\leq\,{\textstyle {d(s)/4M}}\label{eq-proof2}\\
{\textstyle {\sum_{k=r_{1}}^{r_{2}-1}{\textstyle {\sum_{\ixv\in\IvecsNotInvolving{s}}w^{k}(\ixv)}}}} & \,>\,{\textstyle {d(s)/4M}}\label{eq-proof3}
\end{align}
Note that $r_{2}$ must exist, since the series $\sum_{k=0}^{\infty}{\textstyle {\sum_{\ixv\in\IvecsNotInvolving{s}}w^{k}(\ixv)}}$
diverges. It now follows from (\ref{eq:pre}) and (\ref{eq-proof2})
that the conditions of Lemma~\ref{lem-fund} hold. \nop From this,
(\ref{eq-proof3}), and the conclusion of Lemma~\ref{lem-fund} we
see that (\ref{eq:post}) holds for $\kappa'=r_{2}$, whence (\ref{eq:post})
also holds for all $\kappa'>r_{2}$ (by Lemma~\ref{lem-fact-one}).
\end{proof}

\begin{corollary} \label{cor-fund} Suppose $x_{0}\not\in\bigcap\mathcal{C}$,
so that $M>0$. Let $s\in\R^{n}\setminus\bigcap\mathcal{C}$ be such
that $d(s)=d(s,\bigcup_{i\in\IxSetNotContaining{s}}C_{i})>0$ and
the series $\sum_{k=0}^{\infty}{\textstyle {\sum_{\ixv\in\IvecsNotInvolving{s}}w^{k}(\ixv)}}$
diverges. Then there are at most finitely many iteration numbers $k$
such that $\|x^{k}-s\|\leq{\textstyle {d(s)/4}}$. \end{corollary}

\begin{proof} Suppose for contradiction that $\|x^{k}-s\|\leq{\textstyle {d(s)/4}}$
for infinitely many iteration numbers $k$. Let $c$ be any point
in $\bigcap\mathcal{C}$. Then we see from Corollary~\ref{cor-fund0}
that for every iteration number $k$ such that $\|x^{k}-s\|\leq{\textstyle {d(s)/4}}$
there exists an iteration number $k'>k$ such that both of the following
are true: 
\begin{gather*}
\|x^{k'}-s\|\leq d(s)/4\\
\|x^{k'}-c\|^{2}\leq\|x^{k}-c\|^{2}-d(s)^{3}/16M
\end{gather*}
Hence there is an infinite sequence $k_{0}<k_{1}<k_{2}<\dots$ such
that the following are true for every $j\geq1$: 
\begin{gather}
\|x^{k_{j}}-s\|\leq{\textstyle {d(s)/4}\notag}\nonumber \\
\|x^{k_{j}}-c\|^{2}\leq\|x^{k_{j-1}}-c\|^{2}-{\textstyle {d(s)^{3}/16M}}\label{eq:impossible}
\end{gather}
But (\ref{eq:impossible}) cannot hold for all $j\geq1$ because $d(s)^{3}/16M$
is a positive value that is the same for every $j$ and $\|x^{k_{j}}-c\|^{2}$
is nonnegative for all $j$. \end{proof}

\subsection{Proof of the Third Convergence Theorem}

\label{subsec-UASCproof}

We see from Lemma~\ref{lem-fact-one} that $x^{0},x^{1},x^{2},\dots$
is a bounded sequence in $\R^{n}$, and must therefore have a convergent
subsequence. Let $s$ be the limit of a convergent subsequence of
$x^{0},x^{1},x^{2},\dots$.

If $s\in\bigcap\mathcal{C}$, then it follows from Lemma~\ref{first-claim}
that $x^{0},x^{1},x^{2},\dots$ converges to $s$ and the Third Convergence
Theorem holds.

So let us now assume $s\notin\bigcap\mathcal{C}$. Under the hypotheses
of the theorem, $\textbf{UASC}(\mathcal{C})$ holds and therefore
$d(s)>0$. As $s$ is the limit of a convergent subsequence of $x^{0},x^{1},x^{2},\dots$
and $d(s)>0$, it follows from Corollary~\ref{cor-fund} that the
series $\sum_{k=0}^{\infty}\sum_{\ixv\in\IvecsNotInvolving{s}}w^{k}(\ixv)$
cannot diverge. So we see from Corollary~\ref{cor-converge} that
the sequence $x^{0},x^{1},x^{2},\dots$ converges to $s$.

Since $\IvecsNotInvolving{s}\!\supseteq\!\Ivecs\langle i\rangle$
for every index $i$ in $\IxSetNotContaining{s}$, the fact that the
series $\sum_{k=0}^{\infty}\sum_{\ixv\in\IvecsNotInvolving{s}}w^{k}(\ixv)$
does not diverge implies there is no $i\in\IxSetNotContaining{s}$
for which the series $\sum_{k=0}^{\infty}\sum_{\ixv\in\Ivecs\langle i\rangle}w^{k}(\ixv)$
diverges. Equivalently, there is no index $i$ such that $s\notin C_{i}$
for which the series $\sum_{k=0}^{\infty}\sum_{\ixv\in\Ivecs\langle i\rangle}w^{k}(\ixv)$
diverges. \hfill{}\qedsymbol

\section{Fourth Convergence Theorem\label{sec:Fourth-Convergence-Theorem}}

We now consider perturbed versions of sequences that satisfy equation
(\ref{alg}). For each iteration number $k$, let $w^{k}:\Ivecs\rightarrow[0,1]$
be a weight function, let $\bm{v}^{k}$ be a vector in $\R^{n}$,
and let $T_{k}:\R^{n}\rightarrow\R^{n}$ be defined by
\[
T_{k}(x)={\textstyle \sum_{\ixv\in\Ivecs}w^{k}(\ixv)P[\ixv](x)\qquad\text{for all \ensuremath{x\in\R^{n}.}}}
\]
Let $x_{*}^{0}$ be an arbitrary point in $\R^{n}$, and let $x_{*}^{1},x_{*}^{2},x_{*}^{3},\dots\in\R^{n}$
be defined by 
\begin{equation}
{\textstyle {x_{*}^{k+1}=T_{k}(x_{*}^{k})+\bm{v}^{k}=\left(\sum_{\ixv\in\Ivecs}w^{k}(\ixv)P[\ixv](x_{*}^{k})\right)+\bm{v}^{k}}\qquad\text{for every \ensuremath{k\geq0}}.}\label{algP2nd}
\end{equation}
Our next convergence theorem says that, if the series $\sum_{i=0}^{\infty}\|\bm{v}^{k}\|$
converges, then our three earlier convergence theorems remain valid
if we substitute $x_{*}^{0},x_{*}^{1},x_{*}^{2},x_{*}^{3},\dots$
for $x^{0},x^{1},x^{2},x^{3},\dots$:

\begin{theorem}[Fourth Convergence Theorem] \label{3rdconv} Suppose
the series $\sum_{k=0}^{\infty}\|\bm{v}^{k}\|$ converges. Then: 
\begin{lyxlist}{00.000}
\item [{\hspace{2em}1.}] If $\bigcap\mathcal{C}$ is an $n$-dimensional
subset of $\R^{n}$ or \emph{$\textbf{UASC}(\mathcal{C})$} holds,
then the sequence $x_{*}^{0},x_{*}^{1},x_{*}^{2},\dots$ converges,
and its limit lies in $C_{j}$ for every index $j$ such that $\sum_{k=0}^{\infty}\sum_{\ixv\in\Ivecs\langle j\rangle}w^{k}(\ixv)=\infty$.
In particular, if $\bigcap\mathcal{C}$ is an $n$-dimensional subset
of $\R^{n}$ or \emph{$\textbf{UASC}(\mathcal{C})$} holds, and $\sum_{k=0}^{\infty}\sum_{\ixv\in\Ivecs\langle j\rangle}w^{k}(\ixv)=\infty$
for every index $j$ such that $C_{j}\neq\R^{n}$, then $x_{*}^{0},x_{*}^{1},x_{*}^{2},\dots$
converges to a point in $\bigcap\mathcal{C}$.
\item [{\hspace{2em}2.}] \noindent If the weight functions $w^{0},w^{1},w^{2},\dots$
satisfy conditions \textbf{C1} and \textbf{C2} of the Second Convergence
Theorem for some $f_{1},f_{2},f_{3},\dots\in[0,1]$ such that $\sum_{r=1}^{\infty}f_{r}=\infty$
and a sequence of iteration numbers $\kappa_{1}<\kappa_{2}<\kappa_{3}<\dots$,
then the sequence $x_{*}^{0},x_{*}^{1},x_{*}^{2},\dots$ converges
to a point in $\bigcap\mathcal{C}$. 
\end{lyxlist}
\end{theorem}

Important examples of iterates satisfying (\ref{algP2nd}) for perturbation
vectors $(\bm{v}^{k})_{k=0}^{\infty}$ such that $\sum_{k=0}^{\infty}\|\bm{v}^{k}\|$
converges are the sequences $(x_{**}^{k})_{k=0}^{\infty}$ defined
by 
\begin{equation}
{\textstyle {x_{**}^{k+1}=T_{k}(x_{**}^{k}+\beta_{k}\bm{u}^{k})=\sum_{\ixv\in\Ivecs}w^{k}(\ixv)P[\ixv](x_{**}^{k}+\beta_{k}\bm{u}^{k})}\qquad\text{for every \ensuremath{k\geq0}}}\label{eq:superiorization}
\end{equation}

\noindent where $(\bm{u}^{k})_{k=0}^{\infty}$ may be any bounded
sequence of vectors and $(\beta_{k})_{k=0}^{\infty}$ any sequence
of scalars such that $\sum_{k=0}^{\infty}|\beta_{k}|<\infty$: Indeed,
if (\ref{eq:superiorization}) holds and we define $\bm{v}^{k}=T_{k}(x_{**}^{k}+\beta_{k}\bm{u}^{k})-T_{k}(x_{**}^{k})$,
then (\ref{algP2nd}) holds with $(x_{**}^{k})_{k=0}^{\infty}$ in
place of $(x_{*}^{k})_{k=0}^{\infty}$ and it will follow from Lemma~\ref{lem:non-expansive-1}
below that $\|\bm{v}^{k}\|\leq\|\beta_{k}\bm{u}^{k}\|$ for every
$k$ (whence $\sum_{k=0}^{\infty}\|\bm{v}^{k}\|$ converges). The
\emph{superiorization} methodology \cite{Cens15,CHJ17} derives iterative
schemes of the form (\ref{eq:superiorization}) from iterative schemes
of the form (\ref{alg}) by choosing vectors $\beta_{k}\bm{u}^{k}$
that steer the sequence of iterates with the aim of improving the
efficacy of the iterative scheme.

The Fourth Convergence Theorem is an immediate consequence of statements
1 and 3 of Proposition~\ref{prop-3rdConvC} below. Our proof of that
proposition will depend on the next three lemmas. We will give a self-contained
argument that deduces Proposition~\ref{prop-3rdConvC} (and hence
our Fourth Convergence Theorem) from these lemmas. But when $\sum_{k=0}^{\infty}\sum_{\ixv\in\Ivecs\langle j\rangle}w^{k}(\ixv)=\infty$
for every index $j$ such that $C_{j}\neq\R^{n}$ (which is always
true under the hypotheses of statement 2 of the theorem) the theorem
follows at once from these lemmas and a result of Butnariu, Reich,
and Zaslavski \cite[Thm. 3.2]{BuRZ08}, while Proposition~\ref{prop-3rdConvC}
follows at once from the lemmas and \cite[Thm. 11]{BaRZ18}. 

For all pairs of nonnegative integers $(k,l)$ such that $k\geq l$,
we define the point $x_{l}^{k}\in\R^{n}$ as follows: 
\begin{alignat}{2}
x_{0}^{0} & \ =\ x_{*}^{0}\label{eq:x00}\\
x_{l}^{k+1} & \ =\ {\textstyle {T_{k}(x_{l}^{k})}} & \qquad & \text{whenever \ensuremath{k\geq l\geq0}}.\label{eq:xlk+1}\\
x_{k+1}^{k+1} & \ =\ x_{k}^{k+1}+\bm{v}^{k} & \qquad & \text{for all \ensuremath{k\geq0}}.\label{eq:xk+1k+1}
\end{alignat}
\nop From (\ref{eq:xk+1k+1}) and the case $l=k$ of (\ref{eq:xlk+1})
we see that $x_{k+1}^{k+1}\ =\ {\textstyle {T_{k}(x_{k}^{k})+\bm{v}^{k}}}$
for all $k\geq0$. It follows from this, (\ref{eq:x00}), and (\ref{algP2nd}),
that: 
\begin{equation}
x_{k}^{k}\ =\ x_{*}^{k}\qquad\text{for all \ensuremath{k\geq0}.}\label{eq:xkk}
\end{equation}
 Figure~\ref{relationships} shows how different points $x_{l}^{k}$
relate to each other.

\begin{figure}
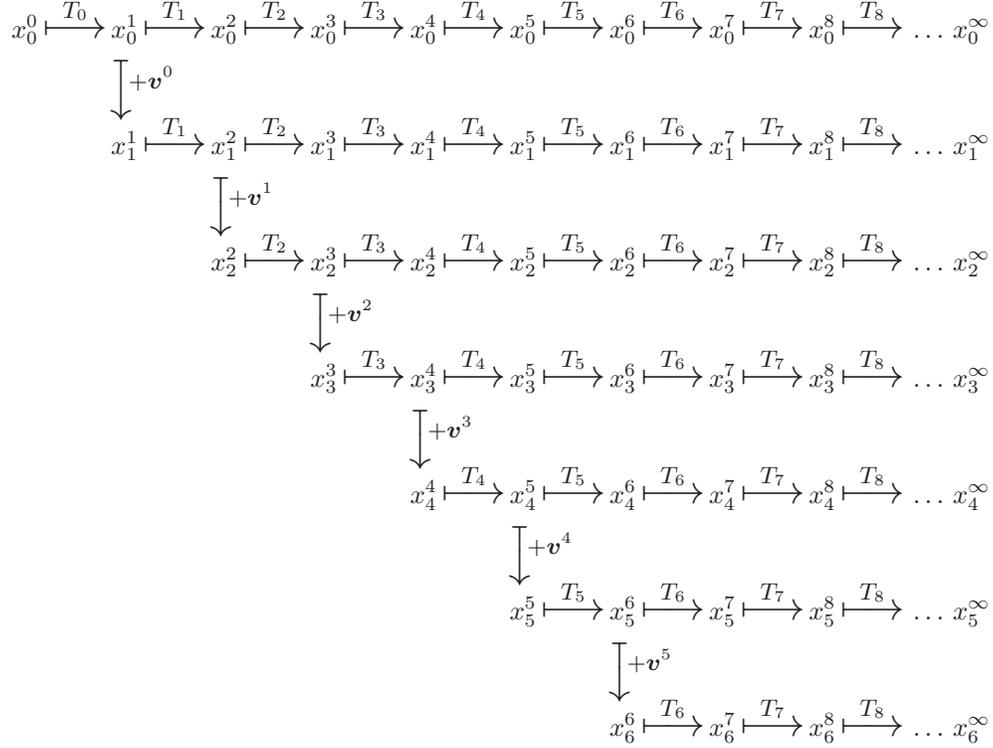

\begin{alignat*}{10}
 & x_{0}^{0}\wmapsto{\text{\small\ensuremath{T_{0}}}} &  & x_{0}^{1}\wmapsto{\text{\small\ensuremath{T_{1}}}} &  & x_{0}^{2}\wmapsto{\text{\small\ensuremath{T_{2}}}} &  & x_{0}^{3}\wmapsto{\text{\small\ensuremath{T_{3}}}} &  & x_{0}^{4}\wmapsto{\text{\small\ensuremath{T_{4}}}} &  & x_{0}^{5}\wmapsto{\text{\small\ensuremath{T_{5}}}} &  & x_{0}^{6}\wmapsto{\text{\small\ensuremath{T_{6}}}} &  & x_{0}^{7}\wmapsto{\text{\small\ensuremath{T_{7}}}} &  & x_{0}^{8}\wmapsto{\text{\small\ensuremath{T_{8}}}}\dots\  &  & x_{0}^{\infty}\\
 &  &  & \downmapsto{\text{\small\ensuremath{+\bm{v}^{0}}}}\\
 &  &  & x_{1}^{1}\wmapsto{\text{\small\ensuremath{\ensuremath{T_{1}}}}} &  & x_{1}^{2}\wmapsto{\text{\small\ensuremath{T_{2}}}} &  & x_{1}^{3}\wmapsto{\text{\small\ensuremath{T_{3}}}} &  & x_{1}^{4}\wmapsto{\text{\small\ensuremath{\ensuremath{T_{4}}}}} &  & x_{1}^{5}\wmapsto{\text{\small\ensuremath{T_{5}}}} &  & x_{1}^{6}\wmapsto{\text{\small\ensuremath{T_{6}}}} &  & x_{1}^{7}\wmapsto{\text{\small\ensuremath{T_{7}}}} &  & x_{1}^{8}\wmapsto{\text{\small\ensuremath{T_{8}}}}\dots\  &  & x_{1}^{\infty}\\
 &  &  &  &  & \downmapsto{\text{\small\ensuremath{+\bm{v}^{1}}}}\\
 &  &  &  &  & x_{2}^{2}\wmapsto{\text{\small\ensuremath{T_{2}}}} &  & x_{2}^{3}\wmapsto{\text{\small\ensuremath{T_{3}}}} &  & x_{2}^{4}\wmapsto{\text{\small\ensuremath{\ensuremath{T_{4}}}}} &  & x_{2}^{5}\wmapsto{\text{\small\ensuremath{T_{5}}}} &  & x_{2}^{6}\wmapsto{\text{\small\ensuremath{T_{6}}}} &  & x_{2}^{7}\wmapsto{\text{\small\ensuremath{T_{7}}}} &  & x_{2}^{8}\wmapsto{\text{\small\ensuremath{T_{8}}}}\dots\  &  & x_{2}^{\infty}\\
 &  &  &  &  &  &  & \downmapsto{\text{\small\ensuremath{+\bm{v}^{2}}}}\\
 &  &  &  &  &  &  & x_{3}^{3}\wmapsto{\text{\small\ensuremath{T_{3}}}} &  & x_{3}^{4}\wmapsto{\text{\small\ensuremath{\ensuremath{T_{4}}}}} &  & x_{3}^{5}\wmapsto{\text{\small\ensuremath{T_{5}}}} &  & x_{3}^{6}\wmapsto{\text{\small\ensuremath{T_{6}}}} &  & x_{3}^{7}\wmapsto{\text{\small\ensuremath{T_{7}}}} &  & x_{3}^{8}\wmapsto{\text{\small\ensuremath{T_{8}}}}\dots\  &  & x_{3}^{\infty}\\
 &  &  &  &  &  &  &  &  & \downmapsto{\text{\small\ensuremath{+\bm{v}^{3}}}}\\
 &  &  &  &  &  &  &  &  & x_{4}^{4}\wmapsto{\text{\small\ensuremath{\ensuremath{T_{4}}}}} &  & x_{4}^{5}\wmapsto{\text{\small\ensuremath{T_{5}}}} &  & x_{4}^{6}\wmapsto{\text{\small\ensuremath{T_{6}}}} &  & x_{4}^{7}\wmapsto{\text{\small\ensuremath{T_{7}}}} &  & x_{4}^{8}\wmapsto{\text{\small\ensuremath{T_{8}}}}\dots\  &  & x_{4}^{\infty}\\
 &  &  &  &  &  &  &  &  &  &  & \downmapsto{\text{\small\ensuremath{+\bm{v}^{4}}}}\\
 &  &  &  &  &  &  &  &  &  &  & x_{5}^{5}\wmapsto{\text{\small\ensuremath{T_{5}}}} &  & x_{5}^{6}\wmapsto{\text{\small\ensuremath{T_{6}}}} &  & x_{5}^{7}\wmapsto{\text{\small\ensuremath{T_{7}}}} &  & x_{5}^{8}\wmapsto{\text{\small\ensuremath{T_{8}}}}\dots\  &  & x_{5}^{\infty}\\
 &  &  &  &  &  &  &  &  &  &  &  &  & \downmapsto{\text{\small\ensuremath{+\bm{v}^{5}}}}\\
 &  &  &  &  &  &  &  &  &  &  &  &  & x_{6}^{6}\wmapsto{\text{\small\ensuremath{T_{6}}}} &  & x_{6}^{7}\wmapsto{\text{\small\ensuremath{T_{7}}}} &  & x_{6}^{8}\wmapsto{\text{\small\ensuremath{T_{8}}}}\dots\  &  & x_{6}^{\infty}
\end{alignat*}
 \caption{\label{relationships} Relationships between different points $x_{l}^{k}$.
Note that $x_{*}^{k}=x_{k}^{k}$ for all $k\geq0$. Each point $x_{l}^{\infty}$
denotes the limit of the sequence $(x_{l}^{l+k})_{k=0}^{\infty}$,
which is convergent under the hypotheses of either part of the Fourth
Convergence Theorem (by Lemmas~\ref{lem-3rdConvA} and \ref{lem-3rdConvB}).}
\end{figure}

\begin{lemma} \label{lem-3rdConvA} Suppose at least one of the following
is true: 
\begin{lyxlist}{00.000}
\item [{\hspace{2.25em}1.}] $\bigcap\mathcal{C}$ is an $n$-dimensional
subset of $\R^{n}$. 
\item [{\hspace{2.25em}2.}] \noindent \emph{$\textbf{UASC}(\mathcal{C})$}
holds. 
\end{lyxlist}
Then for every nonnegative integer $l$ the sequence $\left((T_{l+k}\circ\dots\circ T_{l})(x_{l}^{l})\right)_{k=0}^{\infty}=x_{l}^{l+1},x_{l}^{l+2},x_{l}^{l+3},\dots$
converges, and its limit lies in $C_{j}$ for every index $j$ such
that $\sum_{k=0}^{\infty}\sum_{\ixv\in\Ivecs\langle j\rangle}w^{k}(\ixv)=\infty$.
\end{lemma}

\begin{proof} Let $l$ be any nonnegative integer. We see from Figure~\ref{relationships}
that equation (\ref{alg}) holds when we substitute $x_{l}^{l},x_{l}^{l+1},x_{l}^{l+2},x_{l}^{l+3},\dots$
and $w^{l},w^{l+1},w^{l+2},w^{l+3},\dots$ for $x^{0},x^{1},x^{2},x^{3},\dots$
and $w^{0},w^{1},w^{2},w^{3},\dots$. So it follows from the First
and the Third Convergence Theorems that the sequence $x_{l}^{l},x_{l}^{l+1},x_{l}^{l+2},x_{l}^{l+3},\dots$
converges, and that its limit lies in $C_{j}$ for every index $j$
such that $\sum_{k=0}^{\infty}\sum_{\ixv\in\Ivecs\langle j\rangle}w^{l+k}(\ixv)=\infty$.
As $\sum_{k=0}^{\infty}\sum_{\ixv\in\Ivecs\langle j\rangle}w^{l+k}(\ixv)=\infty$
just if $\sum_{k=0}^{\infty}\sum_{\ixv\in\Ivecs\langle j\rangle}w^{k}(\ixv)=\infty$,
the lemma follows. \end{proof}

\begin{lemma} \label{lem-3rdConvB} Suppose the sequence of weight
functions $(w^{k})_{k=0}^{\infty}$ satisfies conditions \textbf{C1}
and \textbf{C2} of the Second Convergence Theorem for some sequence
$(f_{r})_{r=1}^{\infty}$ of numbers in $[0,1]$ such that $\sum_{r=1}^{\infty}f_{r}=\infty$
and some strictly increasing sequence of iteration numbers $(\kappa_{r})_{r=1}^{\infty}$.
Then for every nonnegative integer $l$ the sequence $\left((T_{l+k}\circ\dots\circ T_{l})(x_{l}^{l})\right)_{k=0}^{\infty}=x_{l}^{l+1},x_{l}^{l+2},x_{l}^{l+3},\dots$
converges to a point in $\bigcap\mathcal{C}$. \end{lemma}

\begin{proof} Let $l$ be any nonnegative integer, and $\rho(l)$
any nonnegative integer such that $\kappa_{\rho(l)+1}\geq l$. For
all $r\in\Z^{+},$ we will write $f_{r}^{\langle l\rangle}$ for $f_{\rho(l)+r}$
and write $\kappa_{r}^{\langle l\rangle}$ for $\kappa_{\rho(l)+r}-l$.
Note that $(f_{r}^{\langle l\rangle})_{r=1}^{\infty}$ is a sequence
of numbers in $[0,1]$ such that $\sum_{r=1}^{\infty}f_{r}^{\langle l\rangle}=\infty$.
Moreover, we see from the definition of $\rho(l)$ that $(\kappa_{r}^{\langle l\rangle})_{r=1}^{\infty}$
is a strictly increasing sequence of nonnegative integers (i.e., iteration
numbers). 

\begin{claim*} The sequence $(w^{l+k})_{k=0}^{\infty}$ satisfies
conditions \textbf{C1} and \textbf{C2} of the Second Convergence Theorem
for the sequence $(f_{r}^{\langle l\rangle})_{r=1}^{\infty}$ and
the sequence of iteration numbers $(\kappa_{r}^{\langle l\rangle})_{r=1}^{\infty}$.
\end{claim*}

As we observed in the proof of the previous lemma, equation (\ref{alg})
continues to hold when we substitute $w^{l},w^{l+1},w^{l+2},w^{l+3},\dots$
and $x_{l}^{l},x_{l}^{l+1},x_{l}^{l+2},x_{l}^{l+3},\dots$ for $w^{0},w^{1},w^{2},w^{3},\dots$
and $x^{0},x^{1},x^{2},x^{3},\dots$. So if the Claim is valid then
it will follow from the Second Convergence Theorem that $x_{l}^{l},x_{l}^{l+1},x_{l}^{l+2},x_{l}^{l+3},\dots$
converges to a point in $\bigcap\mathcal{C}$ and the lemma will be
proved. To complete the proof, we proceed to justify the Claim.

For all $i,r\in\Z^{+}$, let $\textbf{C1}[\bm{w}_{0},\bm{f},\bm{\kappa}](r)$,
$\textbf{C1}[\bm{w}_{l},\bm{f}^{\langle l\rangle},\bm{\kappa}^{\langle l\rangle}](r)$,
$\textbf{C2}[\bm{w}_{0},\bm{f},\bm{\kappa}](i,r)$, and $\textbf{C2}[\bm{w}_{l},\bm{f}^{\langle l\rangle},\bm{\kappa}^{\langle l\rangle}](i,r)$
be the conditions defined as follows: 
\begin{alignat*}{1}
\makebox[85pt][l]{\textbf{C1}[\ensuremath{\bm{w}_{0},\bm{f},\bm{\kappa}](r})} & \quad\Leftrightarrow{\textstyle {\quad\sum_{k=\kappa_{r}}^{\kappa_{r+1}-2}\maxLength(w^{k})\leq\frac{1}{f_{r}}}}\text{\,\,\ if\,\,}\ensuremath{f_{r}>0}\\
\makebox[85pt][l]{\textbf{C1}[\ensuremath{\bm{w}_{l},\bm{f}^{\langle l\rangle},\bm{\kappa}^{\langle l\rangle}](r})} & \quad\Leftrightarrow{\textstyle {\quad\sum_{k=\kappa_{r}^{\langle l\rangle}}^{\kappa_{r+1}^{\langle l\rangle}-2}\maxLength(w^{l+k})\leq\frac{1}{f_{r}^{\langle l\rangle}}}}\,\,\text{\ if\,\,\ensuremath{f_{r}^{\langle l\rangle}>0}}\\
\makebox[85pt][l]{\textbf{C2}[\ensuremath{\bm{w}_{0},\bm{f},\bm{\kappa}](i,r})} & \quad\Leftrightarrow{\textstyle {\quad{\displaystyle \max_{\kappa_{r}\leq k<\kappa_{r+1}}}\sum_{\ixv\in\Ivecs\langle i\rangle}\frac{w^{k}(\ixv)}{\Position(i,\ixv)}\geq f_{r}}}\\
\makebox[85pt][l]{\textbf{C2}[\ensuremath{\bm{w}_{l},\bm{f}^{\langle l\rangle},\bm{\kappa}^{\langle l\rangle}](i,r})} & \quad\Leftrightarrow{\quad{\textstyle {\displaystyle \max_{\kappa_{r}^{\langle l\rangle}\leq k<\kappa_{r+1}^{\langle l\rangle}}}\sum_{\ixv\in\Ivecs\langle i\rangle}\frac{w^{l+k}(\ixv)}{\Position(i,\ixv)}\geq f_{r}^{\langle l\rangle}}}
\end{alignat*}
Then the hypothesis that $(w^{k})_{k=0}^{\infty}$ satisfies conditions
\textbf{C1} and \textbf{C2} of the Second Convergence Theorem for
the sequences $(f_{r})_{r=1}^{\infty}$ and $(\kappa_{r})_{r=1}^{\infty}$
is equivalent to the following assertions:
\begin{lyxlist}{00....}
\item [{$\text{\quad}1:$}] \noindent $\textbf{C1}[\bm{w}_{0},\bm{f},\bm{\kappa}](r)$
holds for all $r\in\Z^{+}$.
\item [{\quad{}2:}] For all $i\in\Z^{+}$ such that $C_{i}\neq\R^{n}$,
$\textbf{C2}[\bm{w}_{0},\bm{f},\bm{\kappa}](i,r)$ holds for all sufficiently
large $r\in\Z^{+}$. 
\end{lyxlist}
Assertions 1 and 2 imply the assertions obtained from them by replacing
$\textbf{C1}[\bm{w}_{0},\bm{f},\bm{\kappa}](r)$ and $\textbf{C2}[\bm{w}_{0},\bm{f},\bm{\kappa}](i,r)$
with $\textbf{C1}[\bm{w}_{0},\bm{f},\bm{\kappa}](\rho(l)+r)$ and
$\textbf{C2}[\bm{w}_{0},\bm{f},\bm{\kappa}](i,\rho(l)+r).$ The resulting
assertions are equivalent to:
\begin{lyxlist}{00....}
\item [{\quad{}1$'$:}] \noindent ${\textstyle \text{For all \ensuremath{r\in\Z^{+},}\quad}\sum_{k=\kappa_{\rho(l)+r}}^{\kappa_{\rho(l)+r+1}-2}\maxLength(w^{k})\leq\frac{1}{f_{\rho(l)+r}}}\text{\,\ if\,\,}\ensuremath{f_{\rho(l)+r}>0}$.
\item [{\quad{}2$'$:}] \noindent For all $i\in\Z^{+}$ such that $C_{i}\neq\R^{n}$,
${\displaystyle \max_{\kappa_{\rho(l)+r}\leq k<\kappa_{\rho(l)+r+1}}}\sum_{\ixv\in\Ivecs\langle i\rangle}\frac{w^{k}(\ixv)}{\Position(i,\ixv)}\geq f_{\rho(l)+r}$
for all sufficiently large $r\in\Z^{+}$. 
\end{lyxlist}
\noindent Since the sequence $(w^{k})_{\kappa_{\rho(l)+r}\leq k<\kappa_{\rho(l)+r+1}}$
is the same as $(w^{l+k})_{\kappa_{r}^{\langle l\rangle}\leq k<\kappa_{r+1}^{\langle l\rangle}}$,
and since $f_{\rho(l)+r}=f_{r}^{\langle l\rangle}$ for all $r\in\Z^{+},$
assertions 1$'$ and 2$'$ are respectively equivalent to:
\begin{lyxlist}{00....}
\item [{\quad{}1$''$:}] \noindent ${\textstyle \ensuremath{\textbf{C1}[\ensuremath{\bm{w}_{l},\bm{f}^{\langle l\rangle},\bm{\kappa}^{\langle l\rangle}](r})}}$
holds for all $r\in\Z^{+}$.
\item [{\quad{}2$''$:}] For all $i\in\Z^{+}$ such that $C_{i}\neq\R^{n}$,
$\textbf{C2}[\ensuremath{\bm{w}_{l},\bm{f}^{\langle l\rangle},\bm{\kappa}^{\langle l\rangle}](i,r})$
holds for all sufficiently large $r\in\Z^{+}$. 
\end{lyxlist}
As assertions 1$''$ and 2$''$ are together equivalent to our Claim,
the proof is complete. \end{proof}

Our next lemma is equivalent to a result that was stated and used
in \cite{BDHK07} and is a generalization of the following well-known
fact, which says that the operators $P_{i}$ are nonexpansive: 

\begin{fact} \label{fact:non-expansive-1} For any index $i$ and
all $y,z\in\R^{n}$, $\|P_{i}(y)-P_{i}(z)\|\leq\|y-z\|$. \end{fact}

Fact~\ref{fact:non-expansive-1} follows from Thm.~3.6 in \cite[Ch. 1]{GoRe84};
it can be deduced from the fact that if $P_{i}(y)\neq P_{i}(z)$ then
(as a consequence of Lemma~\ref{lem-geom}) $y$ and $z$ must lie
on opposite sides of the open set bounded by the hyperplanes orthogonal
to the vector $P_{i}(y)-P_{i}(z)$ that contain $P_{i}(y)$ and $P_{i}(z)$. 

\begin{lemma} \label{lem:non-expansive-1} Let $y,z\in\R^{n}$ and
let $w$ be an arbitrary weight function. Then: 
\[
{\textstyle {\|\sum_{\ixv\in\Ivecs}w(\ixv)P[\ixv](y)-\sum_{\ixv\in\Ivecs}w(\ixv)P[\ixv](z)\|\ \leq\ \|y-z\|}}
\]

\noindent In particular, we have that $\|T_{k}(y)-T_{k}(z)\|\leq\|y-z\|$
for all iteration numbers $k$.

\end{lemma}

\begin{proof} It follows easily from Fact~\ref{fact:non-expansive-1}
that: 
\begin{equation}
\|P[\ixv](y)-P[\ixv](z)\|\ \leq\ \|y-z\|\quad\text{for all \ensuremath{\ixv\in\Ivecs}.}\label{eq:nonexpansive-1}
\end{equation}
Hence we have that 
\begin{alignat*}{2}
{\textstyle {\|\sum_{\ixv\in\Ivecs}w(\ixv)P[\ixv](y)}} & -{\textstyle {\sum_{\ixv\in\Ivecs}w(\ixv)P[\ixv](z)\|}}\\
 & \quad=\ {\textstyle {\|\sum_{\ixv\in\Ivecs}w(\ixv)(P[\ixv](y)-P[\ixv](z))\|}}\\
 & \quad\leq\ {\textstyle {\sum_{\ixv\in\Ivecs}w(\ixv)\|P[\ixv](y)-P[\ixv](z)\|}\qquad} & \text{by the triangle inequality}\\
 & \quad\leq\ {\textstyle {\sum_{\ixv\in\Ivecs}w(\ixv)\|y-z\|}\qquad} & \text{by (\ref{eq:nonexpansive-1})}\\
 & \quad=\ \|y-z\|\qquad & \text{as \ensuremath{{\textstyle {\sum_{\ixv\in\Ivecs}w(\ixv)=1}}}}
\end{alignat*}
which proves the lemma. \end{proof}

When the hypotheses of Lemma~\ref{lem-3rdConvA} or Lemma~\ref{lem-3rdConvB}
hold and $l$ is a nonnegative integer, we define $x_{l}^{\infty}$
by:
\[
x_{l}^{\infty}=\lim_{k\rightarrow\infty}(T_{k-1}\circ\dots\circ T_{l})x_{*}^{l}=\lim_{k\rightarrow\infty}x_{l}^{k}
\]
The existence of the limit $x_{l}^{\infty}$ is guaranteed by the
two lemmas. Now we are ready to state and prove the following proposition,
which contains our Fourth Convergence Theorem: 

\begin{prop} \label{prop-3rdConvC} Suppose at least one of the following
is true:
\begin{lyxlist}{..0.}
\item [{(i)}] Either $\bigcap\mathcal{C}$ is an $n$-dimensional subset
of $\R^{n}$, or $\textbf{UASC}(\mathcal{C})$ holds.
\item [{(ii)}] $w^{0},w^{1},w^{2},\dots$ satisfy conditions \textbf{C1}
and \textbf{C2} of the Second Convergence Theorem for some sequence
$\mbox{\ensuremath{f_{1},f_{2},f_{3},\dots\in[0,1]}}$ such that $\sum_{r=1}^{\infty}f_{r}=\infty$
and some sequence of iteration numbers $\kappa_{1}<\kappa_{2}<\kappa_{3}<\dots$. 
\end{lyxlist}
Then there exists a point $x_{\infty}^{\infty}$ such that all of
the following are true: 
\begin{lyxlist}{00.00...}
\item [{\hspace{2.25em}1.}] \noindent $\lim_{k\rightarrow\infty}x_{*}^{k}=\lim_{l\rightarrow\infty}x_{l}^{\infty}=x_{\infty}^{\infty}$ 
\item [{\hspace{2.25em}2.}] \noindent For $k\geq l\geq0$, $\|x_{*}^{k}-x_{\infty}^{\infty}\|\leq\|(T_{k-1}\circ\dots\circ T_{l})x_{*}^{l}-x_{l}^{\infty}\|+2\sum_{j=l}^{\infty}\|\bm{v}^{j}\|=\|x_{l}^{k}-x_{l}^{\infty}\|+2\sum_{j=l}^{\infty}\|\bm{v}{}^{j}\|$. 
\item [{\hspace{2.25em}3.}] \noindent If (i) holds, then $x_{\infty}^{\infty}\in C_{j}$
for every index $j$ that satisfies $\sum_{k=0}^{\infty}\sum_{\ixv\in\Ivecs\langle j\rangle}w^{k}(\ixv)=\infty$.
If (ii) holds, then $x_{\infty}^{\infty}\in\bigcap\mathcal{C}$.
\end{lyxlist}
\end{prop}

\begin{proof}As $x_{l+1}^{l+1}-x_{l}^{l+1}=\bm{v}^{l}$ for all $l\geq0$,
we see from (\ref{eq:xlk+1}) (or Fig.~\ref{relationships}) and
Lemma~\ref{lem:non-expansive-1} that 
\begin{equation}
\|x_{l}^{k}-x_{l+1}^{k}\|\leq\|\bm{v}^{l}\|\quad\text{whenever \ensuremath{k>l\geq0}}.
\end{equation}
\nop From this and the triangle inequality we deduce that 
\begin{equation}
\|x_{l}^{k}-x_{m}^{k}\|={\textstyle {\|\sum_{j=l}^{m-1}(x_{j}^{k}-x_{j+1}^{k})\|}\leq{\sum_{j=l}^{m-1}\|x_{j}^{k}-x_{j+1}^{k}\|}}\leq{\textstyle {\sum_{j=l}^{m-1}\|\bm{v}^{j}\|}\,\,\,\text{whenever\,\,\,}\ensuremath{k\geq m\geq l\geq0}}.\label{eq:another_bound-1}
\end{equation}
It follows from (\ref{eq:another_bound-1}) that: 
\begin{alignat}{2}
 & \|x_{l}^{\infty}-x_{m}^{\infty}\|\ \leq\ {\textstyle {\sum_{j=l}^{m-1}\|\bm{v}^{j}\|\ \leq\ \sum_{j=l}^{\infty}\|\bm{v}^{j}\|}} & \text{} & \text{\text{\qquad{whenever\,\,\,\text{\ensuremath{m\geq l\geq0}.}}}}\label{eq:chain1}\\
 & \|x_{*}^{k}-x_{l}^{k}\|\,=\,\|x_{k}^{k}-x_{l}^{k}\|\,\leq\,{\textstyle \sum_{j=l}^{k-1}\|\bm{v}^{j}\|\ \leq\ \sum_{j=l}^{\infty}\|\bm{v}^{j}\|} &  & \text{\qquad{whenever\,\,\,\text{\ensuremath{k\geq l\geq0}.}}}\label{eq:chain2}\\
 & \|x_{*}^{k}-x_{m}^{\infty}\|\ \leq{\textstyle \ \|x_{l}^{k}-x_{l}^{\infty}\|+2\sum_{j=l}^{\infty}\|\bm{v}^{j}\|} &  & \text{\qquad{whenever\,\,\,\text{\ensuremath{\min(m,k)\geq l\geq0}.}}}\label{eq:chain3}
\end{alignat}
Here (\ref{eq:chain3}) follows from (\ref{eq:chain1}), (\ref{eq:chain2}),
and the triangle inequality. Since $\sum_{j=l}^{\infty}\|\bm{v}^{j}\|\rightarrow0$
as $l\rightarrow\infty$, we see from (\ref{eq:chain1}) that $(x_{l}^{\infty})_{l=0}^{\infty}$
is a Cauchy sequence. Let $x_{\infty}^{\infty}$ be the limit of this
sequence, so the second equality of statement 1 holds. Then statement
2 also holds, because of (\ref{eq:chain3}) and the fact that $\lim_{m\rightarrow\infty}x_{m}^{\infty}=x_{\infty}^{\infty}$. 

For every $\epsilon>0$, let $l(\epsilon)$ and $K(\epsilon)$ be
integers such that $\sum_{j=l(\epsilon)}^{\infty}\|\bm{v}^{j}\|<\epsilon/3$
and $\|x_{l(\epsilon)}^{k}-x_{l(\epsilon)}^{\infty}\|<\epsilon/3$
whenever $k>K(\epsilon)$. Then for every $\epsilon>0$ we see (on
putting $l=l(\epsilon)$ in statement 2) that $\|x_{*}^{k}-x_{\infty}^{\infty}\|<\epsilon$
whenever $k>K(\epsilon)$. This establishes statement 1.

In case (i), Lemma~\ref{lem-3rdConvA} tells us $x_{l}^{\infty}\in C_{j}$
for every $l$ and every index $j$ such that $\sum_{k=0}^{\infty}\sum_{\ixv\in\Ivecs\langle j\rangle}w^{k}(\ixv)=\infty$.
This and statement 1 imply $x_{\infty}^{\infty}\in C_{j}$ for every
index $j$ such that $\sum_{k=0}^{\infty}\sum_{\ixv\in\Ivecs\langle j\rangle}w^{k}(\ixv)=\infty$
(as $C_{j}$ is closed for every $j$). In case (ii), Lemma~\ref{lem-3rdConvB}
tells us $x_{l}^{\infty}\in\bigcap\mathcal{C}$ for every $l$. This
and statement 1 imply $x_{\infty}^{\infty}\in\bigcap\mathcal{C}$
(as $\bigcap\mathcal{C}$ is closed). Hence statement 3 holds.\end{proof}

As mentioned earlier, our Fourth Convergence Theorem follows immediately
from statements 1 and 3 of the above proposition.

Statement 2 of the proposition is equivalent to the case of \cite[eqn. (4.5)]{BaRZ18}
in which the Hilbert space $\mathcal{H}$, sequence $(e_{k})_{k=i}^{\infty}$,
and points $x^{k+1}$, $x^{i}$, and $x^{\infty}$ of \cite[eqn. (4.5)]{BaRZ18}
are respectively $\R^{n}$, our sequence $(\|\bm{v}^{k}\|)_{k=i}^{\infty}$,
and our points $x_{*}^{k+1}$, $x_{*}^{i}$, and $x_{\infty}^{\infty}$.
(However, in case (i) of the proposition our points $x_{l}^{\infty}$
need not satisfy the hypothesis of \cite[eqn. (4.5)]{BaRZ18} that
these points be common fixed points of all the $T_{k}$s.) It follows
from statement 2 that, under the hypotheses of either part of the
Fourth Convergence Theorem, the perturbations of (\ref{algP2nd})
preserve the rate of convergence of the iterative scheme (\ref{alg})
in a sense that is explained and illustrated in Remark 13 and Example
16 of \cite{BaRZ18}. This applies, in particular, to perturbations
introduced by superiorization of the iterative scheme (\ref{alg})
\cite[Cor. 14]{BaRZ18}. 

\section{Fifth Convergence Theorem}

Recall from sec.~\ref{subsec:Outline} that the operators $P_{i}^{\epsilon}$
and $P^{\epsilon}[\ixv]$ are defined for every $\epsilon>0$, every
index $i$, and every index vector $\ixv=(\ixv_{1},\ixv_{2},\dots,\ixv_{m})$
as follows: 
\begin{align*}
P_{i}^{\epsilon}(x) & =\begin{cases}
P_{i}(x) & \text{if \ensuremath{\|x-P_{i}(x)\|\geq\epsilon}}\\
x & \text{if \ensuremath{\|x-P_{i}(x)\|<\epsilon}}
\end{cases}\\
P^{\epsilon}[\ixv] & =P_{\ixv_{m}}^{\epsilon}\circ P_{\ixv_{m-1}}^{\epsilon}\circ\dots\circ P_{\ixv_{1}}^{\epsilon}
\end{align*}
Our final convergence theorem deals with sequences that are generated
using the operators $P^{\epsilon}[\ixv]$ rather than the operators
$P[\ixv]$.

\begin{theorem}[Fifth Convergence Theorem] \label{Fifth} Let $x_{\epsilon}^{0}$
be any point in $\R^{n}$, and let $x_{\epsilon}^{1},x_{\epsilon}^{2},x_{\epsilon}^{3}\dots\in\R^{n}$
be the sequence of points defined by 
\begin{equation}
{\textstyle {x_{\epsilon}^{k+1}=\left(\sum_{\ixv\in\Ivecs}w^{k}(\ixv)P^{\epsilon}[\ixv](x_{\epsilon}^{k})\right)+\bm{v}^{k}}\qquad\text{for all \ensuremath{k\geq0}},}\label{algEpsPerturb2}
\end{equation}
where $\bm{v}^{0},\bm{v}^{1},\bm{v}^{2},\dots$ is a sequence of vectors
in $\R^{n}$ such that $\sum_{k=0}^{\infty}\|\bm{v}^{k}\|$ converges.
Then the sequence $x_{\epsilon}^{1},x_{\epsilon}^{2},x_{\epsilon}^{3},\dots$
converges. Moreover, if $p$ is the limit of this sequence then $d(p,C_{j})\leq\epsilon$
for every index $j$ such that $\sum_{k=0}^{\infty}\sum_{\ixv\in\Ivecs\langle j\rangle}w^{k}(\ixv)=\infty$.
\end{theorem}

We mention here that Fact~\ref{fact:non-expansive-1} does not hold
when $P_{i}$ is replaced by $P_{i}^{\epsilon}$: When $C_{i}\neq\R^{n}$
the operator $P_{i}^{\epsilon}$ is \emph{not} non-expansive; indeed,
$P_{i}^{\epsilon}$ is not even continuous at those points $x$ for
which $d(x,C_{i})=\epsilon$. So the argument we used to prove the
Fourth Convergence Theorem cannot be used to deduce the Fifth Convergence
Theorem from the special case of the theorem in which $\bm{v}^{k}=0$
for every $k$.

\subsection{Proof of the Fifth Convergence Theorem}

For every weight function $w$, index vector $\ixv$, and point $x\in\R^{n}$
we define the following complementary sets of index vectors: 
\begin{align*}
\IvecsEpsAllInvolving{x} & \ =\ \{\ixv\in\Ivecs\mid d(x,C_{j})<\epsilon\text{ for all \ensuremath{j\in\IxSetOfVec(\ixv)}}\}\\
\IvecsEpsNotInvolving{x} & \ =\ \{\ixv\in\Ivecs\mid d(x,C_{j})\geq\epsilon\text{ for at least one \ensuremath{j\in\IxSetOfVec(\ixv)}}\}\ =\ \Ivecs\setminus\IvecsEpsAllInvolving{x}
\end{align*}

\begin{lemma} \label{lem-triv-bound5} Let $c$ be a point in $\bigcap\mathcal{C}$,
$x$ a point in $\R^{n}$, and $i$ an index. Then: 
\begin{lyxlist}{00.000}
\item [{\hspace{2.25em}1.}] $\|P_{i}^{\epsilon}(x)-c\|^{2}\leq\|x-c\|^{2}-\epsilon^{2}$
if $d(x,C_{i})\geq\epsilon$. 
\item [{\hspace{2.25em}2.}] $\|P_{i}^{\epsilon}(x)-c\|\leq\|x-c\|$, with
equality if and only if $d(x,C_{i})<\epsilon$. 
\end{lyxlist}
\end{lemma}

\begin{corollary} \label{cor-upperbound5} Let $c$ be a point in
$\bigcap\mathcal{C}$, $x$ a point in $\R^{n}$, and $\ixv$ an index
vector. Then we have that $\|P^{\epsilon}[\ixv](x)-c\|^{2}\leq\|x-c\|^{2}-\epsilon^{2}$
if $\ixv\in\IvecsEpsNotInvolving{x}$. \hfill{}\qedsymbol \end{corollary}

\begin{corollary} \label{cor-triv-bound5} Let $c$ be a point in
$\bigcap\mathcal{C}$, $x$ a point in $\R^{n}$, and $\ixv$ an index
vector. Then $\|P^{\epsilon}[\ixv](x)-c\|\leq\|x-c\|$, with equality
if and only if $\ixv\in\IvecsEpsAllInvolving{x}$. \hfill{}\qedsymbol
\end{corollary}

\begin{proof}[Proof of Lemma~\ref{lem-triv-bound5}] Assertion 1
follows from Proposition~\ref{fact-one}, because if $d(x,C_{i})\geq\epsilon$
then $P_{i}^{\epsilon}(x)=P_{i}(x)$.

If $d(x,C_{i})<\epsilon$, then $\|P_{i}^{\epsilon}(x)-c\|=\|x-c\|$
(because $P_{i}^{\epsilon}(x)=x$) and so assertion 2 is true. If
$d(x,C_{i})\geq\epsilon$, then assertion 1 implies $\|P_{i}^{\epsilon}(x)-c\|<\|x-c\|$
and so assertion 2 is again true. \end{proof}

Using Corollary~\ref{cor-triv-bound5}, we show:

\begin{lemma} \label{lem-bound5} Let $c$ be a point in $\bigcap\mathcal{C}$.
Then: 
\begin{alignat}{2}
\|x_{\epsilon}^{k+1}-c\| & \ \leq\ \|x_{\epsilon}^{k}-c\|+\|\bm{v}^{k}\| & \quad & \text{for any iteration number \ensuremath{k}.}\\
\|x_{\epsilon}^{k}-c\| & \ \leq\ \|x_{\epsilon}^{0}-c\|+{\textstyle {\sum_{\ell=0}^{k-1}\|\bm{v}^{\ell}\|}} & \quad & \text{for any iteration number \ensuremath{k}.}\label{eq:2nd-leq}
\end{alignat}
\end{lemma}

\begin{proof} The first inequality holds because 
\begin{alignat*}{2}
\|x_{\epsilon}^{k+1}-c\| & \ =\ {\textstyle {\left\Vert \left(\sum_{\ixv\in\Ivecs}w^{k}(\ixv)P^{\epsilon}[\ixv](x_{\epsilon}^{k})\right)+\bm{v}^{k}-c\right\Vert }}\\
 & \ \leq\ \|\bm{v}^{k}\|+{\textstyle {\left\Vert \left(\sum_{\ixv\in\Ivecs}w^{k}(\ixv)P^{\epsilon}[\ixv](x_{\epsilon}^{k})\right)-c\right\Vert }} & \qquad & \text{by the triangle inequality}\\
 & \ \leq\ \|\bm{v}^{k}\|+{\textstyle {\sum_{\ixv\in\Ivecs}w^{k}(\ixv)\|P^{\epsilon}[\ixv](x_{\epsilon}^{k})-c\|}} & \qquad & \text{as\,\ \ensuremath{z\mapsto\|z-c\|}\,\,\ is convex}\\
 & \ \leq\ \|\bm{v}^{k}\|+{\textstyle {\sum_{\ixv\in\Ivecs}w^{k}(\ixv)\|x_{\epsilon}^{k}-c\|}} & \qquad & \text{by Corollary~\ref{cor-triv-bound5}}\\
 & \ =\ \|\bm{v}^{k}\|+\|x_{\epsilon}^{k}-c\| & \qquad & \text{as \ensuremath{{\textstyle {\sum_{\ixv\in\Ivecs}w^{k}(\ixv)}=1}}}
\end{alignat*}
The second inequality follows from the first. \end{proof}

In the sequel $M_{\star}$ will denote the nonnegative number defined
by 
\[
{\textstyle {M_{\star}=2(d(x_{\epsilon}^{0},\bigcap\mathcal{C})+\sum_{k=0}^{\infty}\|\bm{v}^{k}\|})}
\]
We observe that $M_{\star}$ depends only on the seed point $x_{\epsilon}^{0}$,
the collection $\mathcal{C}$, and the sum $\sum_{k=0}^{\infty}\|\bm{v}^{k}\|$
of the lengths of the perturbation vectors.

\begin{lemma} \label{lem:M-bound} Let $\ixv$ be any index vector
and $k$ any iteration number. Then: 
\[
\|P^{\epsilon}[\ixv](x_{\epsilon}^{k})+\bm{v}^{k}-x_{\epsilon}^{k}\|\leq M_{\star}
\]
\end{lemma}

\begin{proof} For every $c\in\bigcap\mathcal{C}$ we have that: 
\begin{alignat*}{2}
\|P^{\epsilon}[\ixv](x_{\epsilon}^{k})+\bm{v}^{k}-x_{\epsilon}^{k}\| & \ \leq\|\bm{v}^{k}\|+\|P^{\epsilon}[\ixv](x_{\epsilon}^{k})-c\|+\|c-x_{\epsilon}^{k}\| & \qquad & \text{by the triangle inequality}\\
 & \ \leq\ \|\bm{v}^{k}\|+\|x_{\epsilon}^{k}-c\|+\|c-x_{\epsilon}^{k}\| & \qquad & \text{by Corollary~\ref{cor-triv-bound5}}\\
 & \ =\ \|\bm{v}^{k}\|+2\|x_{\epsilon}^{k}-c\|\\
 & \ \leq\ \|\bm{v}^{k}\|+2(\|x_{\epsilon}^{0}-c\|+{\textstyle {\sum_{\ell=0}^{k-1}\|\bm{v}^{\ell}\|)}} & \qquad & \text{by Lemma~\ref{lem-bound5}}\\
 & \ \leq\ 2(\|x_{\epsilon}^{0}-c\|+{\textstyle {\sum_{\ell=0}^{\infty}\|\bm{v}^{\ell}\|)}}
\end{alignat*}
Hence $\|P^{\epsilon}[\ixv](x_{\epsilon}^{k})+\bm{v}^{k}-x_{\epsilon}^{k}\|\leq\inf_{c\in\mathcal{C}}2(\|x_{\epsilon}^{0}-c\|+\sum_{\ell=0}^{\infty}\|\bm{v}^{\ell}\|)=M_{\star}$.
\end{proof}

The sums $\sum_{\ixv\in\IvecsEpsAllInvolving{x_{\epsilon}^{k}}}w^{k}(\ixv)$
and $\sum_{\ixv\in\IvecsEpsNotInvolving{x_{\epsilon}^{k}}}w^{k}(\ixv)$
will be much used in the sequel. These are respectively the sum of
the weights at iteration $k$ of those operators $P^{\epsilon}[\ixv]$
for which $P^{\epsilon}[\ixv](x_{\epsilon}^{k})=x_{\epsilon}^{k}$,
and the sum of the weights at iteration $k$ of those operators $P^{\epsilon}[\ixv]$
for which $P^{\epsilon}[\ixv](x_{\epsilon}^{k})\neq x_{\epsilon}^{k}$.
(As the functions $w^{0},w^{1},w^{2},\dots$ need not be fixed in
advance, $w^{k}$ may be chosen after computing $P^{\epsilon}[\ixv](x_{\epsilon}^{k})$
for various index vectors $\ixv$, in which case it would be quite
natural to choose $w^{k}$ so that either $w^{k}(\ixv)=0$ for all
$\ixv\in\IvecsEpsAllInvolving{x_{\epsilon}^{k}}$ and $\sum_{\ixv\in\IvecsEpsNotInvolving{x_{\epsilon}^{k}}}w^{k}(\ixv)=1$,
or $x_{\epsilon}^{k+1}=x_{\epsilon}^{k}$. However, we will not assume
this.)

For any iteration numbers $r_{1}<r_{2}$ and any $s$ in $\R^{n}$,
the next lemma gives an upper bound for $\|x_{\epsilon}^{r_{2}}-x_{\epsilon}^{r_{1}}\|$
in terms of $\sum_{k=r_{1}}^{r_{2}-1}\sum_{\ixv\in\IvecsEpsNotInvolving{x_{\epsilon}^{k}}}w^{k}(\ixv)$.

\begin{lemma} \label{lem-newer} Let $r_{1}$ and $r_{2}$ be iteration
numbers such that $r_{1}<r_{2}$, and let $s\in\R^{n}$. Then we have
that $\|x_{\epsilon}^{r_{2}}-x_{\epsilon}^{r_{1}}\|\leq{\textstyle {\sum_{k=r_{1}}^{r_{2}-1}\|\bm{v}^{k}\|+M_{\star}\sum_{k=r_{1}}^{r_{2}-1}\sum_{\ixv\in\IvecsEpsNotInvolving{x_{\epsilon}^{k}}}w^{k}(\ixv)}}$.
\end{lemma}

\begin{proof} For any iteration number $k$, it follows from the
convexity of $z\mapsto\|z+u\|$ (for any $u\in\R^{n}$), the fact
that $P^{\epsilon}[\ixv](x_{\epsilon}^{k})=x_{\epsilon}^{k}$ when
$\ixv\in\IvecsEpsAllInvolving{x_{\epsilon}^{k}}$, and Lemma~\ref{lem:M-bound}
that: 
\begin{alignat*}{2}
\|x_{\epsilon}^{k+1}-x_{\epsilon}^{k}\| & \ =\ {\textstyle {\left\Vert \left(\sum_{\ixv\in\Ivecs}w^{k}(\ixv)P^{\epsilon}[\ixv](x_{\epsilon}^{k})\right)+\bm{v}^{k}-x_{\epsilon}^{k}\right\Vert }\notag}\\
 & \ \leq\ {\textstyle {\sum_{\ixv\in\Ivecs}w^{k}(\ixv)\|P^{\epsilon}[\ixv](x_{\epsilon}^{k})+\bm{v}^{k}-x_{\epsilon}^{k}\|}} &  & \text{\quad as\,\ \ensuremath{z\mapsto\|z+\bm{v}^{k}-x_{\epsilon}^{k}\|}\,\,\ is convex}\\
 & \ =\ {\textstyle {\sum_{\ixv\in\IvecsEpsAllInvolving{x_{\epsilon}^{k}}}w^{k}(\ixv)\|P^{\epsilon}[\ixv](x_{\epsilon}^{k})+\bm{v}^{k}-x_{\epsilon}^{k}\|}\notag}\\
 & \qquad\quad+\ {\textstyle {\sum_{\ixv\in\IvecsEpsNotInvolving{x_{\epsilon}^{k}}}w^{k}(\ixv)\|P^{\epsilon}[\ixv](x_{\epsilon}^{k})+\bm{v}^{k}-x_{\epsilon}^{k}\|}\notag}\\
 & \ \leq\ {\textstyle {\sum_{\ixv\in\IvecsEpsAllInvolving{x_{\epsilon}^{k}}}w^{k}(\ixv)\|\bm{v}^{k}\|}+{\textstyle {\sum_{\ixv\in\IvecsEpsNotInvolving{x_{\epsilon}^{k}}}w^{k}(\ixv)M_{\star}}}} &  & \text{\quad by Lemma~\ref{lem:M-bound}}\\
 & \ \leq\ \|\bm{v}^{k}\|\ +{\textstyle {\sum_{\ixv\in\IvecsEpsNotInvolving{x_{\epsilon}^{k}}}}w^{k}(\ixv)M_{\star}} &  & \text{\quad as \ensuremath{{\textstyle {\sum_{\ixv\in\IvecsEpsAllInvolving{x_{\epsilon}^{k}}}w^{k}(\ixv)}\leq1}}}
\end{alignat*}
The lemma follows from this and the triangle inequality. \end{proof}

\begin{lemma} \label{lem-upperbound6} Let $c$ be any point in $\bigcap\mathcal{C}$.
Then there exists a number $A$ that is independent of $k$ such that
${\|x_{\epsilon}^{k+1}-c\|^{2}\leq\|x_{\epsilon}^{k}-c\|^{2}-\epsilon^{2}\sum_{\ixv\in\IvecsEpsNotInvolving{x_{\epsilon}^{k}}}w^{k}(\ixv)+A\|\bm{v}^{k}\|}\,\,\text{for every iteration number \ensuremath{k}}$.\end{lemma}

\begin{corollary} \label{cor:bound on rate of decrease} Let $c$
be any point in $\bigcap\mathcal{C}$, Then there exists a number
$A$ that is independent of $\kappa$ such that ${\|x_{\epsilon}^{\kappa}-c\|^{2}\leq\|x_{\epsilon}^{0}-c\|^{2}-\epsilon^{2}\sum_{k=0}^{\kappa-1}\sum_{\ixv\in\IvecsEpsNotInvolving{x_{\epsilon}^{k}}}w^{k}(\ixv)+A\sum_{k=0}^{\infty}\|\bm{v}^{k}\|}$~~for
every iteration number $\kappa$. \hfill{}\qedsymbol \end{corollary}

\begin{corollary} \label{cor:convergence4} The series $\sum_{k=0}^{\infty}\sum_{\ixv\in\IvecsEpsNotInvolving{x_{\epsilon}^{k}}}w^{k}(\ixv)$
converges. \hfill{}\qedsymbol \end{corollary}
\begin{proof}
[Proof of Lemma~\ref{lem-upperbound6}] Let $A=2(\|x_{\epsilon}^{0}-c\|+{\textstyle {\sum_{\ell=0}^{\infty}\|\bm{v}^{\ell}\|})}$,
and let $k$ be any iteration number. Then: 
\begin{alignat}{2}
\|x_{\epsilon}^{k+1}-c\|^{2} & \ =\ {\textstyle {\left\Vert \left(\sum_{\ixv\in\Ivecs}w^{k}(\ixv)P^{\epsilon}[\ixv](x_{\epsilon}^{k})\right)+\bm{v}^{k}-c\right\Vert ^{2}}\notag}\nonumber \\
 & \ \leq\ {\textstyle {\sum_{\ixv\in\Ivecs}w^{k}(\ixv)\|P^{\epsilon}[\ixv](x_{\epsilon}^{k})+\bm{v}^{k}-c\|^{2}}} & \ \  & \text{as\,\ \ensuremath{z\mapsto\|z+\bm{v}^{k}-c\|^{2}}\,\,\ is convex}\nonumber \\
 & \ \leq\ {\textstyle {\sum_{\ixv\in\Ivecs}w^{k}(\ixv)(\|\bm{v}^{k}\|+\|P^{\epsilon}[\ixv](x_{\epsilon}^{k})-c\|)^{2}}} & \ \  & \text{by the triangle inequality}\nonumber \\
 & \ =\ {\textstyle {\sum_{\ixv\in\Ivecs}w^{k}(\ixv)(\|\bm{v}^{k}\|^{2}+2\|\bm{v}^{k}\|\|P^{\epsilon}[\ixv](x_{\epsilon}^{k})-c\|)}\notag}\nonumber \\
 & \qquad+{\textstyle {\sum_{\ixv\in\Ivecs}w^{k}(\ixv)\|P^{\epsilon}[\ixv](x_{\epsilon}^{k})-c\|^{2}}}\label{eq:47}
\end{alignat}
The first of the two sums on the right side of (\ref{eq:47}) satisfies
\begin{alignat}{2}
{\textstyle {\sum_{\ixv\in\Ivecs}w^{k}(\ixv)}(\|\bm{v}^{k}\|^{2}} & +2\|\bm{v}^{k}\|\|P^{\epsilon}[\ixv](x_{\epsilon}^{k})-c\|)\nonumber \\
 & \ \leq\ {\textstyle {\sum_{\ixv\in\Ivecs}w^{k}(\ixv)(\|\bm{v}^{k}\|^{2}+2\|\bm{v}^{k}\|\|x_{\epsilon}^{k}-c\|})} & \quad & \text{by Corollary~\ref{cor-triv-bound5}}\nonumber \\
 & \ =\ \|\bm{v}^{k}\|^{2}+2\|\bm{v}^{k}\|\|x_{\epsilon}^{k}-c\| & \quad & \text{as \ensuremath{{\textstyle {\sum_{\ixv\in\Ivecs}w^{k}(\ixv)}=1}}}\nonumber \\
 & \ =\ (\|\bm{v}^{k}\|+2\|x_{\epsilon}^{k}-c\|)\|\bm{v}^{k}\|\nonumber \\
 & \ \leq\ \left(\|\bm{v}^{k}\|+2(\|x_{\epsilon}^{0}-c\|+{\textstyle {\sum_{\ell=0}^{k-1}\|\bm{v}^{\ell}\|})}\right)\|\bm{v}^{k}\| & \quad & \text{by (\ref{eq:2nd-leq})}\nonumber \\
 & \ \leq\ A\|\bm{v}^{k}\|\label{eq:48}
\end{alignat}
The second of the two sums on the right side of (\ref{eq:47}) satisfies
\begin{alignat*}{2}
{\textstyle {\sum_{\ixv\in\Ivecs}w^{k}(\ixv)\|P^{\epsilon}[\ixv](x_{\epsilon}^{k})-c\|^{2}}} & \ =\ {\textstyle {\sum_{\ixv\in\IvecsEpsAllInvolving{x_{\epsilon}^{k}}}w^{k}(\ixv)\|P^{\epsilon}[\ixv](x_{\epsilon}^{k})-c\|^{2}}\notag}\\
 & \qquad+{\textstyle {\sum_{\ixv\in\IvecsEpsNotInvolving{x_{\epsilon}^{k}}}w^{k}(\ixv)\|P^{\epsilon}[\ixv](x_{\epsilon}^{k})-c\|^{2}}\notag}\\
 & \ \leq\ {\textstyle {\sum_{\ixv\in\IvecsEpsAllInvolving{x_{\epsilon}^{k}}}w^{k}(\ixv)\|x_{\epsilon}^{k}-c\|^{2}}\notag}\\
 & \qquad+{\textstyle {\sum_{\ixv\in\IvecsEpsNotInvolving{x_{\epsilon}^{k}}}w^{k}(\ixv)(\|x_{\epsilon}^{k}-c\|^{2}-\epsilon^{2})}} & \ \  & \text{by Corollary~\ref{cor-upperbound5}}\\
 & \ =\ {\textstyle {\|x_{\epsilon}^{k}-c\|^{2}-\epsilon^{2}\sum_{\ixv\in\IvecsEpsNotInvolving{x_{\epsilon}^{k}}}w^{k}(\ixv)}} & \ \  & \text{as \ensuremath{{\textstyle {\sum_{\ixv\in\Ivecs}w^{k}(\ixv)=1}}}}
\end{alignat*}
The lemma follows from this, (\ref{eq:47}), and (\ref{eq:48}). 
\end{proof}

\subsubsection*{Completion of the Proof of the Fifth Convergence Theorem}

As $\sum_{k=0}^{\infty}\sum_{\ixv\in\IvecsEpsNotInvolving{x_{\epsilon}^{k}}}w^{k}(\ixv)$
converges (by Corollary~\ref{cor:convergence4}), it follows from
Lemma~\ref{lem-newer} that $x_{\epsilon}^{0},x_{\epsilon}^{1},x_{\epsilon}^{2},\dots$
is a Cauchy sequence and must therefore converge.

Let $p$ be the limit of the sequence, and let $j$ be any index such
that $d(p,C_{j})>\epsilon$. To complete the proof of the theorem,
we show that the series $\sum_{k=0}^{\infty}\sum_{\ixv\in\Ivecs\langle j\rangle}w^{k}(\ixv)$
converges.

Since $p$ is the limit of $x_{\epsilon}^{0},x_{\epsilon}^{1},x_{\epsilon}^{2},\dots$,
we see from the definition of $j$ that $d(x_{\epsilon}^{k},C_{j})>\epsilon$
for all sufficiently large iteration numbers $k$. Hence $\Ivecs\langle j\rangle\subseteq\IvecsEpsNotInvolving{x_{\epsilon}^{k}}$
for all sufficiently large $k$. So the fact that $\sum_{k=0}^{\infty}\sum_{\ixv\in\IvecsEpsNotInvolving{x_{\epsilon}^{k}}}w^{k}(\ixv)$
converges implies $\sum_{k=0}^{\infty}\sum_{\ixv\in\Ivecs\langle j\rangle}w^{k}(\ixv)$
also converges.\hfill{}\qedsymbol

\section{Discussion and Conclusion}

For any infinite sequence of closed convex sets $(C_{i})_{i=1}^{\infty}$
in $\R^{n}$ such that $\bigcap_{i=1}^{\infty}C_{i}\neq\emptyset$,
a \emph{string-averaging algorithm} based on the corresponding projection
operators $(P_{i}:\R^{n}\rightarrow C_{i})_{i=1}^{\infty}$ generates
a sequence of iterates $(x^{k})_{k=0}^{\infty}$ in which each iterate
$x^{k+1}$ is a weighted average of the results of applying certain
finite sequences of those projections (i.e., certain finite sequences
of $P_{i}$s) to $x^{k}$. This paper addresses the problem of finding
sufficient conditions (on $(C_{i})_{i=1}^{\infty}$ and/or on the
averaging weights and the sequences of $P_{i}$s that are used) for
$(x^{k})_{k=0}^{\infty}$ to converge to a point in $\bigcap{}_{i=1}^{\infty}C_{i}$.

In most prior work, string-averaging algorithms have been defined
for a \emph{finite} sequence $(C_{i})_{i=1}^{K}$ of closed convex
sets rather than an infinite sequence. (Such algorithms can be understood
as algorithms for the infinite sequence $(C_{i})_{i=1}^{\infty}$
if we define $C_{i}=\R^{n}$ for all $i>K$.) In this finite case
it is known\textemdash and follows from \cite[Thm. 4.1]{ReZa16}\textemdash that
the iterates $(x^{k})_{k=0}^{\infty}$ converge to a point in $\bigcap_{i=1}^{K}C_{i}$
if the three hypotheses stated as H1 \textendash{} H3 in sec. \ref{subsec:Outline}
are all satisfied. 

The main focus of this paper is on the (countably) infinite case\textemdash i.e.,
the case where there are infinitely many indices $i$ for which $C_{i}\neq\R^{n}$.
In this case it is impossible to satisfy H3, and the hypotheses H1
and H2 would be quite restrictive. Without assuming any of these additional
hypotheses, we have proved three theorems that give sufficient conditions
for the iterates $(x^{k})_{k=0}^{\infty}$ to converge and for the
iterates to converge to a point in $\bigcap_{i=1}^{\infty}C_{i}$. 

From these results we have deduced a fourth theorem that generalizes
the first three theorems to sequences of perturbed iterates for which
the sum of the magnitudes of the perturbations converges. Such perturbed
iterates are produced when an algorithm that generates $(x^{k})_{k=1}^{\infty}$
is superiorized \cite{Cens15,CHJ17} to improve its efficacy. The
fourth theorem implies that those algorithms whose iterates are shown
by our first three theorems to converge to a point in $\bigcap_{i=1}^{\infty}C_{i}$
are also resilient to bounded perturbations, so that they are candidates
for superiorization. When there are only finitely many indices $i$
for which $C_{i}\neq\R^{n}$ and the above-mentioned hypotheses H1
\textendash{} H3 are satisfied, our fourth theorem is a known result
that follows from \cite[Thm. 4.5]{ReZa16}.

Our fifth theorem deals with (possibly perturbed) sequences of iterates
computed by algorithms that use operators $(P_{i}^{\epsilon})_{i=1}^{\infty}$
rather than the operators $(P_{i})_{i=1}^{\infty}$, where $P_{i}^{\epsilon}$
is defined as follows for any $\epsilon>0$: 
\[
P_{i}^{\epsilon}(x)=\begin{cases}
P_{i}(x)\  & \text{if \ensuremath{\|x-P_{i}(x)\|\geq\epsilon}}\\
x & \text{if \ensuremath{\|x-P_{i}(x)\|<\epsilon}}
\end{cases}
\]
Assuming that the sum of the magnitudes of any perturbations converges,
this theorem implies that if the averaging weights satisfy a mild
condition\textemdash namely that, for every index $i$ such that $C_{i}\neq\R^{n}$,
the sum (over all iterations) of the weights of the index vectors
that contain index $i$ is infinite\textemdash then the iterates converge
to a point that is within distance $\epsilon$ of each of the $C_{i}$s.
We can gradually reduce $\epsilon$\textemdash e.g., as explained
at the end of sec.~\ref{subsec:Outline}\textemdash to obtain a sequence
of iterates that can be shown (using Lemma~\ref{lem-bound5}) to
converge to a point in $\bigcap_{i=1}^{\infty}C_{i}$ if the sum of
the magnitudes of any perturbations used in computing this sequence
converges.

\section*{Acknowledgments}

The authors are grateful to a referee for detailed and helpful suggestions
regarding the presentation of the material of sec.~\ref{sec:Fourth-Convergence-Theorem}.
We are also grateful to Elias Helou for drawing our attention to the
fact that one of our mathematical claims needed an additional hypothesis,
and to Yair Censor, Touraj Nikazad, and two of the referees for bringing
a number of references to our attention and for corrections to certain
of our references.


\begin{thebibliography}{123456789}
\bibitem[AlCe05]{AlCe05} A. Aleyner and Y. Censor, Best approximation
to common fixed points of a semigroup of nonexpansive operators, \emph{J.~Nonlinear
Convex Anal.}, Vol.~6, 137\textendash 151, 2005.

\bibitem[BaRZ18]{BaRZ18} C. Bargetz, S. Reich, and R. Zalas, Convergence
properties of dynamic string-averaging projection methods in the presence
of perturbations, \emph{Numer.~Algor.}, Vol.~77, 185\textendash 209,
2018.

\bibitem[BlHe06]{BlHe06} D. Blat and A. O. Hero, III, Energy based
sensor network source localization via projection onto convex sets
(POCS), \textit{IEEE T. Signal Process}., Vol. 54, 3614\textendash 3619,
2006.

\bibitem[BDHK07]{BDHK07} D. Butnariu, R. Davidi, G. T. Herman, and
I. G. Kazantsev, Stable convergence behavior under summable perturbations
of a class of projection methods for convex feasibility and optimization
problems, \emph{IEEE J-STSP}, Vol.~1, 540\textendash 547, 2007.

\bibitem[BuRZ08]{BuRZ08} D. Butnariu, S. Reich, and A.~J.~Zaslavski,
Stable convergence theorems for infinite products and powers of nonexpansive
mappings, \emph{Numer.~Funct.~Anal.~Optim.}, Vol.~29, 304\textendash 323,
2008.

\bibitem[Cens15]{Cens15} Y. Censor, Weak and strong superiorization:
between feasibility-seeking and minimization, An.~\c{S}t.~Univ.~Ovidius
Constan\c{t}a, Vol.~23(3), 41\textendash 64, 2015.

\bibitem[CDH10]{CDH10} Y. Censor, R. Davidi, and G. T. Herman, Perturbation
resilience and superiorization of iterative algorithms, \emph{Inverse
Probl.}, Vol. 26, 065008 (12pp), 2010.

\bibitem[CeEl02]{CeEl02} Y. Censor and T. Elfving, Block-iterative
algorithms with diagonally scaled oblique projections for the linear
feasibility problem, \emph{SIAM J.~Matrix Anal.~Appl.}, Vol. 24,
40\textendash 58, 2002. 

\bibitem[CEH01]{CEH01} Y. Censor, T. Elfving, and G.T. Herman, Averaging
strings of sequential iterations for convex feasibility problems.
\textit{Inherently Parallel Algorithms in Feasibility and Optimization
and their Applications}, D. Butnariu, Y. Censor, and S. Reich (Editors),
Elsevier Science Publishers, Amsterdam, 101\textendash 114, 2001.

\bibitem[CHJ17]{CHJ17} Y. Censor, G. T. Herman, and M. Jiang (Editors),
Special issue on superiorization: theory and applications, \emph{Inverse
Probl}., Vol. 33, No.~4, 2017.

\bibitem[CeZa13]{CeZa13} Y. Censor and A.~J.~Zaslavski, Convergence
and perturbation resilience of dynamic string-averaging projection
methods, \emph{Comput.~Optim.~Appl.}, Vol.~54, 65\textendash 76,
2013.

\bibitem[CeZa15]{CeZa15} Y. Censor and A.~J.~Zaslavski, Strict
Fejér monotonicity by superiorization of feasibility-seeking projection
methods, \emph{J.~Optim.~Theory Appl.}, Vol.~165, 172\textendash 187,
2015.

\bibitem[Cimm38]{Cimm38} G. Cimmino, Calcolo approssimato per le
soluzioni dei sistemi di equazioni lineari, \emph{Ric.\ Sci.\ Progr.\ Tecn.\ Econom.\ Naz.},
Vol.~1, 326\textendash 333, 1938.

\bibitem[Comb97]{Comb97} P. L. Combettes, Convex set theoretic image
recovery by extrapolated iterations of parallel subgradient projections,
\textit{IEEE T.~Image Proces}s., Vol. 6, 493\textendash 506, 1997.

\bibitem[Comb97b]{Comb97b} P. L. Combettes, Hilbertian convex feasibility
problem: convergence of projection methods, \emph{Appl.~Math.~Optim.,}
Vol. 35, 311\textendash 330, 1997.

\bibitem[GoRe84]{GoRe84} K. Goebel and S. Reich, \emph{Uniform Convexity,
Hyperbolic Geometry, and Non-Expansive Mappings}, Marcel Dekker, New
York, 1984.

\bibitem[GPR67]{GPR67} L. G. Gurin (Gubin), B. T. Polyak, and E .V.
Raik, The method of projections for finding the common point of convex
sets, \textit{USSR Comput. Math. Math. Phys.,} Vol. 7, Issue 6, 1\textendash 24,
1967.

\bibitem[Kacz37]{Kacz37} S. Kaczmarz, Angenäherte Auflösung von Systemen
linearer Gleichungen, \emph{Bull.\ Int.\ Acad.\ Polon.\ Sci.\ A},
355\textendash 357, 1937.

\bibitem[KSP09]{KSP09} I. G. Kazantsev, S. Schmidt, and H. F. Poulsen,
A discrete spherical X-ray transform of orientation distribution functions
using bounding cubes, \textit{Inverse Probl.}, Vol 25, 105009 (15pp),
2009.

\bibitem[Nedi10]{Nedi10} A. Nedi{\'c}, Random projection algorithms
for convex set intersection problems, \textit{49th IEEE Conference
on Decision and Control}, 7655\textendash 7660, 2010.

\bibitem[NiAb15]{NiAb15} T. Nikazad and M. Abbasi, Perturbation-resilient
iterative methods with an infinite pool of mappings, \emph{SIAM J.
Numer. Anal.}, Vol. 53, 390\textendash 404, 2015.

\bibitem[NiAb17]{NiAb17} T. Nikazad and M. Abbasi, A unified treatment
of some perturbed fixed point iterative methods with an infinite pool
of operators, \emph{Inverse Probl.}, Vol. 33, 044002 (27pp), 2017.

\bibitem[PhSu13]{PhSu13} W. Phuengrattana and S. Suantai, Common
fixed points of an infinite family of nonexpansive mappings in uniformly
convex metric spaces, \emph{Math. Comput. Model.}, Vol. 57, 306\textendash 310,
2013.

\bibitem[ReZa16]{ReZa16} S. Reich and R. Zalas, A modular string
averaging procedure for solving the common fixed point problem for
quasi-nonexpansive mappings in Hilbert space, \emph{Numer.~Algor.},
Vol.~72, 297\textendash 323, 2016.

\bibitem[Webs94]{Webs94-1} R. J. Webster, \emph{Convexity}, Oxford
University Press, Oxford, 1994.

\bibitem[YaOg04]{YaOg04} I. Yamada and N. Ogura, Adaptive projected
subgradient method for asymptotic minimization of sequence of nonnegative
convex functions, \textit{Numer. Funct. Anal. Optim}., Vol. 25, 593\textendash 617,
2004.

\bibitem[YYZ07]{YYZ07} Y. Yao, J. Yao, and H. Zhou, Approximation
methods for common fixed points of infinite countable family of nonexpansive
mappings, \emph{Comput. Math. Appl.}, Vol. 53, 1380\textendash 1389,
2007.
\end{thebibliography}
\end{document}